\documentclass[12pt]{amsart}

\usepackage{amssymb} 
\usepackage{amsthm}
\usepackage{amscd}
\usepackage{multicol}
\usepackage[dvips]{color}
\usepackage[all]{xy}
\usepackage[dvipdfmx]{hyperref} 
\hypersetup{
    colorlinks=true,   
    linkcolor=blue,    
    citecolor=blue,    
    pdfborder={0 0 0}, 
    bookmarksopen=true 
}

\newtheorem{thm}{Theorem}[section]
\newtheorem{lem}[thm]{Lemma}
\newtheorem{lem-dfn}[thm]{Lemma-Definition}
\newtheorem{prop}[thm]{Proposition}

\theoremstyle{definition}
\newtheorem{defn}[thm]{Definition}

\newtheorem{ex}[thm]{Example}
\newtheorem{nota}[thm]{Notation}
\newtheorem{quest}[thm]{Question}

\theoremstyle{remark}
\newtheorem{rem}[thm]{Remark}
\numberwithin{equation}{section}
\newcommand{\thmref}[1]{Theorem~\ref{#1}}
\newcommand{\lemref}[1]{Lemma~\ref{#1}}

\newcommand{\proref}[1]{Proposition~\ref{#1}}
\newcommand{\remref}[1]{Remark~\ref{#1}}

\newcommand{\figref}[1]{Figure~\ref{#1}}

\DeclareMathOperator{\Supp}{Supp}
\DeclareMathOperator{\Spec}{Spec}
\DeclareMathOperator{\spec}{Spec}

\DeclareMathOperator{\supp}{Supp}
\DeclareMathOperator{\Hom}{Hom}

\DeclareMathOperator{\Img}{Im}

\DeclareMathOperator{\tr}{Tr}
\DeclareMathOperator{\Tr}{Tr}

\DeclareMathOperator{\di}{div}

\DeclareMathOperator{\chr}{char}

\DeclareMathOperator{\cff}{cff}
\newcommand{\m}{\mathfrak m}

\newcommand{\frq}{{\mathfrak q}}
\newcommand{\PP}{\mathbb P}

\newcommand{\Q}{\mathbb Q}

\newcommand{\C}{\mathbb C}

\newcommand{\cF}{\mathcal F}

\newcommand{\cO}{\mathcal O}

\newcommand{\GL}{\mathrm{GL}}

\newcommand{\Z}{Z_f}
\newcommand{\defset}[2]{{\left\{#1\,\left| \,#2 \right. \right\}}}

\begin{document}
\title[Nearly Gorenstein rational surface singularities]{Nearly Gorenstein rational surface singularities}

%%%%%%%%%%%%%%%%%%%%%%%%%%%%%%%%%%%%%%%%%%%%%%%%%%%%%%%%%%%%%%% 
%%   Information for first author
%%%%%%% %%%%%%% %%%%%%% %%%%%%% %%%%%%% %%%%%%% %%%%%%% 
\author{Kyosuke Maeda}
\address[Kyosuke Maeda]{Department of Mathematics, 
College of Humanities and Sciences, 
Nihon University, Setagaya-ku, Tokyo, 156-8550, Japan}
\email{maeda127k@gmail.com}
%%    Information for second author
\author{Tomohiro Okuma}
\address[Tomohiro Okuma]{Department of Mathematical Sciences, 
Yamagata University,  Yamagata, 990-8560, Japan.}
\email{okuma@sci.kj.yamagata-u.ac.jp}
% Orcid# 0000-0001-7590-3923
%%%%%%%%%%%%%%%%%%%%%%%%%%%%%%%%%%%%%%%%%%%%%%%%%%%%%%%%%%%%%%% 
%%    Information for third author
%%%%%%% %%%%%%% %%%%%%% %%%%%%% %%%%%%% %%%%%%% %%%%%%% 
\author{Kei-ichi Watanabe}
\address[Kei-ichi Watanabe]{Department of Mathematics, College of Humanities and Sciences, 
Nihon University, Setagaya-ku, Tokyo, 156-8550, Japan and 
Organization for the Strategic Coordination of Research and Intellectual Properties, Meiji University
}
\email{watnbkei@gmail.com}
% Orcid# 0000-0003-2433-5908
%%%%%%%%%%%%%%%%%%%%%%%%%%%%%%%%%%%%%%%%%%%%%%%%%%%%%%%%%%%%%%% 
%%    Information for fourth author
%%%%%%% %%%%%%% %%%%%%% %%%%%%% %%%%%%% %%%%%%% %%%%%%% 
\author{Ken-ichi Yoshida}
\address[Ken-ichi Yoshida]{Department of Mathematics, 
College of Humanities and Sciences, 
Nihon University, Setagaya-ku, Tokyo, 156-8550, Japan}
\email{yoshida.kennichi@nihon-u.ac.jp}
% Orcid# 0000-0003-3758-5367
%%%%%%%%%%%%%%%%%%%%%%%%%%%%%%%%%%%%%%%%%%%%%%%%%%%%%%%%%%%%%%% 
\thanks{TO was partially supported by JSPS Grant-in-Aid 
for Scientific Research (C) Grant Number 21K03215.
KW  was partially supported by JSPS Grant-in-Aid 
for Scientific Research (C) Grant Number 23K03040.
KY was partially supported by JSPS Grant-in-Aid 
for Scientific Research (C) Grant Number 24K06678.}
%%%%%%%%%%%%%%%%%%%%%%%%%%%%%%%%%%%%%%%%%%%%%%%%%%%%%%%%%%%%%%% 
\date{\today}
\dedicatory{Dedicated to the memory of J\"urgen Herzog}
%\subjclass[2020]{13G05, 13G05, 14J17,13H10, 14J27}
%%%%%%%%%%%%%%%%%%%%%%%%%%%%%%%%%%%%%

%%% Abstract
\begin{abstract} 
In this paper, we show that for any rational surface singularity $A$, 
the canonical trace ideal $\Tr_A(K_A)$ is an integrally closed ideal, 
which is represented 
by the minimal anti-nef cycle $F$ 
on the minimal resolution of singularities 
so that $K_X+F$ is anti-nef. 
Then $F \ge \Z$ if $A$ is not Gorenstein, where $\Z$ is the 
fundamental cycle. 
As a result, we give a criterion for the rational surface singularity $A$ 
to be nearly Gorenstein. 
\par 
Moreover, we classify all nearly Gorenstein rational singularities 
in terms of resolution of singularities  
in the following cases: (a) the fundamental cycle $\Z$ is almost reduced; 
(b) quotient singularities. 
 
\end{abstract}

\maketitle

%%%%%%%%%%%%%%%%%%%%%
\section{Introduction}

First, we introduce the notions of trace ideals and nearly Gorensteinness in a general setting.
Let $(A,\m)$ be a Cohen-Macaulay local ring with maximal ideal $\m$ 
or a positively graded ring with graded maximal ideal $\m$. 
Assume that $A$ has a canonical module $K_A$ and let $Q(A)$ be the total ring of fractions.
For an $A$-module $M$, the {\em trace ideal}
 of $M$ is defined as follows:
\[
\tr_A(M):=\sum_{f\in \Hom_A(M,A)}\Img f= \langle f(x) \,|\, f\in \Hom_A(M,A), \, x\in M \rangle
\]
For any ideal $I\subset A$ of positive grade, 
$\tr_A I = I\cdot I^{-1}$ holds true, where 
$I^{-1}=\defset{x\in Q(A)}{xI\subset A}$ (see \cite[1.1]{HHS}). 
In particular, $\tr_A(K_A)=K_A \cdot K_A^{-1}$. 
\par 
Herzog et al. \cite{HHS} introduced the following notion. 
% 
%%%%% Definition 1.1
\begin{defn}[\textrm{cf. \cite[2.2]{HHS}}]
We call $\Tr_A(K_A)$ the {\em canonical trace ideal} of $A$. 
A local ring $(A,\m)$ is said to be {\em nearly Gorenstein} 
if $\tr_A(K_A) \supset \m$.
\end{defn}

\par 
Note that $A$ is Gorenstein exactly when $\tr_A(K_A) = A$. 
By the definition, Gorenstein rings are nearly Gorenstein, and $A$ is nearly Gorenstein but not Gorenstein if and only if $\tr_A(K_A)=\m$.

\par 
We want to determine the condition for $A$ to be nearly Gorenstein when 
$(A,\m)$ is a {\em normal} local domain of dimension $2$.   
We assume also that $A$ contains an algebraically closed field isomorphic 
to $A/\m$.   
Also, we suppose that there exists a resolution of singularities 
(e.g. $A$ is excellent). We simply call such $A$ a (normal) 
\textit{surface singularity}. 
We assume that $A$ has a canonical module $K_A$.

\par 
We want to determine the condition for $A$ to be nearly Gorenstein 
{\em using the resolution of singularities of $A$}.  
\par  
For a general surface singularity, this problem is rather hard.  But if 
$A$ is a {\em rational singularity}, we can find very simple criterion 
for $A$ to be nearly Gorenstein. See Theorem \ref{Main2}.
\par 
In \cite[Proposition 3.1]{Di}, S. Ding classified 
all two-dimensional nearly Gorenstein quotient singularities 
(see also \cite{CS} for higher-dimensional cyclic quotient singularities). 
However it seems to be not known when a rational singularity is nearly Gorenstein 
even if $\dim A=2$ and there is no general method to compute canonical trace ideals.  
So the main purpose of this paper is to establish a criterion for two-dimensional 
rational singularities to be nearly Gorenstein and give the computation of 
canonical trace ideals in terms of resolution of singularities.

\par \vspace{2mm}
Let us explain the organization of the paper.  
In Sections 2 and 3, we first investigate a geometric description of
(canonical) trace ideals, and 
we show how to compute the canonical trace ideal 
$\Tr_A(K_A)$ for a rational surface singularity $A$. 

%%%%%%%%   Theorem 1.2
\begin{thm}[\textrm{cf. Theorem \ref{rat-cantr}}]  \label{Main1}
Assume that $A$ is a two-dimensional rational singularity, and let $K_A$ be a canonical module of $A$. 
Let $\Z$ denote the fundamental cycle on the minimal resolution $X \to \Spec A$. 
Then $\Tr_A(K_A)$ is an integrally closed $\m$-primary ideal, which is represented on $X$ as follows: 
\[
\Tr_A(K_A)=H^0(X,\mathcal{O}_X(-F)),  
\]
where $F$ is the minimal cycle such that 
$F+K_X$ is anti-nef (cf. \ref{notation} (3)).
Note that $F \ge \Z$ if $K_X \not \equiv 0$ since $F$ is anti-nef. 
\end{thm}

\par \vspace{2mm}
In Section 4, we give a criterion for two-dimensional rational singularity 
to be nearly Gorenstein. 

%%%%%%  Theorem 1.3
\begin{thm}[\textrm{see Theorem \ref{t:nGrat}}] \label{Main2}
Assume that $A$ is a rational singularity, but not Gorenstein.
Then the following conditions are equivalent.
\begin{enumerate}
\item $A$ is nearly Gorenstein.
\item $K_X+\Z$ is anti-nef. 
\item  For every component $E_i$ with $E_i^2\le -3$, we have $\Z E_i \le E_i^2 + 2$.
\end{enumerate}
\end{thm}

\par \vspace{2mm}
In Section 5, we classify two-dimensional nearly Gorenstein 
rational singularities having the almost reduced fundamental 
cycle $\Z$, 
that is, $\cff_{E_i}(\Z)=1$ for any component $E_i$ with $-E_i^2 \ge 3$.  
Our theorem gives an analogue of the well-known 
classification of rational double points (Gorenstein rational singularities).  
As a special case, we can recover the classification of nearly 
Gorenstein rational triple points (see \cite{MY}). 

%%%%%  Theorem 1.4
\begin{thm}[\textrm{cf. Theorem \ref{ARnG}}] \label{Main3}
Assume that the fundamental cycle $\Z$ is almost reduced. 
Then the following conditions are equivalent$:$ 
\begin{enumerate}
\item $A$ is nearly Gorenstein. 
\item The resolution graph of $X$ and $\Z$ is one of the following lists, 
where $\circ$ stands $(-2)$-curve. 
\par 
\begin{figure}[htb]
\begin{picture}(400,50)(-30,0)
    \thicklines
\put(-30,20){\text{Type $A$}}
\put(25,28){{\tiny $1$}}
\put(30,22){\circle*{6}}
\put(33,22){\line(1,0){30}}
%\put(24,10){{\tiny $E_{i_1}$}}
%
\put(55,28){{\tiny $1$}}
\put(61,22){\circle*{6}}
%\put(50,10){{\tiny $-b_{n-1}$}}
%
\put(65,22){\line(1,0){20}}
\put(86,19){$\cdots$}
\put(120,28){{\tiny $1$}}
\put(100,22){\line(1,0){21}}
\put(125,22){\circle*{6}}
%\put(116,10){{\tiny $-b_2$}}
%
\put(147,28){{\tiny $1$}}
\put(128,22){\line(1,0){23}}
\put(155,22){\circle*{6}}
%\put(146,10){{\tiny $E_{i_2}$}}
%
\end{picture}

\begin{picture}(400,50)(-30,0)
    \thicklines
\put(-30,20){\text{Type $D$}}
\put(25,23){{\tiny $1$}}
\put(30,17){\circle*{6}}
\put(33,17){\line(1,0){25}}
%\put(24,5){{\tiny $-b_n$}}
%
\put(55,23){{\tiny $2$}}
\put(61,17){\circle{6}}
%\put(50,5){{\tiny $-b_{n-1}$}}
%  \put(57,5){{\tiny $E_{i_0}$}}
%%
\put(65,17){\line(1,0){20}}
\put(86,14){$\cdots$}
\put(120,23){{\tiny $2$}}
\put(100,17){\line(1,0){21}}
\put(125,17){\circle{6}}
%\put(116,5){{\tiny $-b_1$}}
%
\put(147,23){{\tiny $2$}}
\put(128,17){\line(1,0){23}}
\put(155,17){\circle{6}}
%\put(148,5){{\tiny $-b$}}
%
\put(177,23){{\tiny $1$}}
\put(158,17){\line(1,0){24}}
\put(185,17){\circle*{6}}
\put(146,43){{\tiny $1$}}
\put(155,19){\line(0,1){21}}
\put(155,43){\circle*{6}}
\end{picture}

\begin{picture}(400,45)(-30,0)
    \thicklines
 \put(-30,20){\text{Type $E_6$}}   
  \put(25,18){{\tiny $1$}}
\put(30,12){\circle*{6}}
\put(33,12){\line(1,0){24}}
  \put(55,18){{\tiny $2$}}
\put(60,12){\circle{6}}
\put(63,12){\line(1,0){24}}
  \put(80,18){{\tiny $3$}}
\put(90,12){\circle{6}}
\put(93,12){\line(1,0){24}}
  \put(115,18){{\tiny $2$}}
\put(120,12){\circle{6}}
\put(123,12){\line(1,0){24}}
  \put(145,18){{\tiny $1$}}
\put(150,12){\circle*{6}}
\put(90,16){\line(0,1){15}}
  \put(83,38){{\tiny $2$}}
\put(90,34){\circle{6}}
%  \put(95,35){{\tiny $E_{i_0}$}}
\end{picture}

\begin{picture}(400,55)(30,0)
\thicklines
 \put(30,20){\text{Type $E_7$}}   
\put(88,23){{\tiny $1$}}
\put(97,17){\line(1,0){24}}
\put(93,17){\circle*{6}}
%\put(183,5){{\tiny $$}}
%
\put(118,23){{\tiny $2$}}
\put(127,17){\line(1,0){24}}
\put(123,17){\circle{6}}
%\put(113,5){{\tiny $$}}
%
\put(150,23){{\tiny $3$}}
\put(155,17){\circle{6}}
%\put(146,5){{\tiny $-3$}}
%
\put(177,23){{\tiny $4$}}
%\put(187,23){{$E_0$}}
\put(158,17){\line(1,0){23}}
\put(185,17){\circle{6}}
%\put(178,5){{\tiny $-3$}}
%
\put(207,23){{\tiny $3$}}
\put(188,17){\line(1,0){24}}
\put(215,17){\circle{6}}
%\put(208,5){{\tiny $-3$}}
%
\put(176,43){{\tiny $2$}}
\put(185,19){\line(0,1){21}}
\put(185,43){\circle{6}}
\put(237,23){{\tiny $2$}}
\put(218,17){\line(1,0){24}}
\put(245,17){\circle{6}}
%\put(238,5){{\tiny $-3$}}
%  \put(239,5){{\tiny $E_{i_0}$}}
%
\end{picture}

\begin{picture}(400,50)(-30,0)
    \thicklines
 \put(-30,20){\text{Type $E_8$}}   
  \put(25,18){{\tiny $2$}}
\put(30,12){\circle{6}}
\put(35,12){\line(1,0){20}}
  \put(55,18){{\tiny $3$}}
\put(60,12){\circle{6}}
\put(65,12){\line(1,0){20}}
  \put(83,18){{\tiny $4$}}
\put(90,12){\circle{6}}
\put(95,12){\line(1,0){20}}
  \put(115,18){{\tiny $5$}}
\put(120,12){\circle{6}}
\put(125,12){\line(1,0){20}}
  \put(143,18){{\tiny $6$}}
\put(150,12){\circle{6}}
\put(155,12){\line(1,0){20}}
  \put(175,18){{\tiny $4$}}
\put(180,12){\circle{6}}
\put(185,12){\line(1,0){20}}
  \put(205,18){{\tiny $2$}}
\put(210,12){\circle{6}}
%  \put(205,0){{\tiny $E_{i_0}$}}
%
\put(150,16){\line(0,1){15}}
  \put(142,38){{\tiny $3$}}
\put(150,34){\circle{6}}
\end{picture}
\caption{\label{fig:Center-nG} nG rational singularity of almost reduced fundamental cycle}
\end{figure}
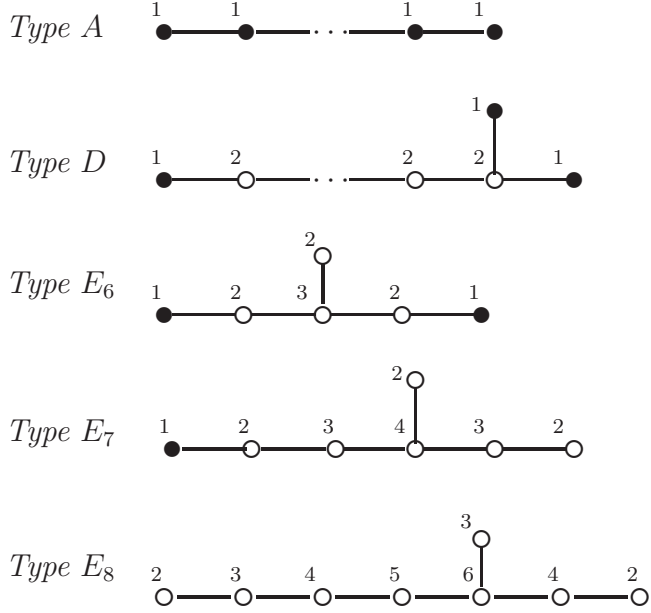
\end{enumerate}

\par
Note that $A$ is Gorenstein if and only if 
all curves are $(-2)$-curves. 
\end{thm}
%%%%%
\par \vspace{2mm}
In Section 6, we classify two-dimensional 
nearly Gorenstein, quotient singularities 
in terms of the Pinkham-Demazure representation 
(or resolution graph). 
Note that our proof gives another proof for
Ding's classification \cite{Di}.

%%%%% Theorem 1.5
\begin{thm}[\textrm{see Theorem $\ref{Quot-nG}$, \cite[Proposition 3.5]{Di}}]  \label{Main4}
Assume $R=R(\PP^1, D)$ is a quotient singularity, not a cyclic quotient singularity  
and not Gorenstein.
If $R$ is nearly Gorenstein, then $D$ is one of the following. 
\begin{enumerate}
\item $1/2 P_1 + 1/2 P_2 - \dfrac{ks - (k-1)}{(k+1) s - k} P_3$ for some $k\ge 0$ and $s\ge 3$.  
\\[1mm]
\item $1/2 P_1 + 2/3 P_2 - 1/3 P_3$
\item $1/2 P_1 + 2/3 P_2 - 2/3 P_3$
\item $1/2 P_1 + 1/3 P_2 - 1/4 P_3$
\item $1/2 P_1 + 2/3 P_2 - 1/4 P_3$
\item $1/2 P_1 + 2/3 P_2 - 3/4 P_3$
\item $1/2 P_1 + 1/3 P_2 - 1/5 P_3$
\item $1/2 P_1 + 1/3 P_2 - 3/5 P_3$
\item $1/2 P_1 + 2/3 P_2 - 1/5 P_3$
\item $1/2 P_1 + 2/3 P_2 - 3/5 P_3$
\item $1/2 P_1 + 2/3 P_2 - 4/5 P_3$
\end{enumerate}
\end{thm}

\par 
Moreover, we show $\ell_A(A/\Tr_A(K_A)) \le e(A)-1$ for 
any quotient singularity $A$; see Proposition \ref{res-quot}. 
Note that $\ell_A(A/\Tr_A(K_A))$ is not bounded above even if
we fix $e(A)$ in general; see Examples \ref{MYTypeA}, \ref{e=4}.

%%%%%%%%%%%%%%%%%%%%%%%%%%%%%%%%%%%%%%%%%%%
%%%%%%%%%%%%%%%%%%%%%%%%%%%%%%%%%%%%%%%%%%%%
%%%%%%%%%%%%%%%%%%%%%%%%%%%%%%%%%%%%%%%%%%%
%%% Section 2
\section{Cycles on a resolution space}

In this section, we set up the notation used throughout the paper and introduce the basic terminology concerning cycles on a resolution of a two-dimensional normal singularity. 

\begin{nota} \label{notation}  
(1) Let $(A, \m, k)$ be a two-dimensional normal local domain containing an algebraically closed field isomorphic to the residue field $k$ and let $K_A$ denote a canonical module of $A$.
Also, we suppose that there exists a resolution of singularities 
  $\pi\:X \to \spec (A)$. Let $E:=\pi^{-1}(\m)$ be the exceptional 
set of $\pi$ 
 and $E=\bigcup_{i=1}^n E_i$ the decomposition into the irreducible components.
Let $K_X$ denote the canonical divisor on $X$.

\par 
(2) The support of a divisor $D$ on $X$ is denoted by $\supp(D)$.
A divisor $D$ on $X$  with $\supp(D)\subset E$ is called a {\em cycle}.
We may consider a cycle $D>0$ as a scheme with structure sheaf $\cO_D := \cO_X/\cO_X(-D)$.
We write $\chi(C)=\chi(\cO_C):=h^0(\cO_C)-h^1(\cO_C)$, where $h^i(\cF)$ denotes $\dim_k H^i(\cF)$.

\par 
(3) Let $L$ be a divisor on $X$.
We say that  $L$ is {\em nef} if $LE_i \ge 0$ for all $E_i \subset E$.
The divisor $L$ is said to be {\em anti-nef} if $-L$ is nef.
It is known that any anti-nef cycle is effective.
\end{nota}

%%% Definition 2.2
\begin{defn}\label{d:numequiv}
For $\Q$-divisors $L_1$ and $L_2$ on $X$, we say that $L_1$ and $L_2$ are {\em numerically equivalent} if $(L_1-L_2)E_i=0$ for every $E_i \subset E$, and write as $L_1\equiv L_2$.
\end{defn}

By the Riemann-Roch formula, we have 
\begin{equation}
\label{eq:RR}
\chi(C)=-C(C+K_X)/2,
\
\chi(\cO_C(L))=\chi(C)+LC.
\end{equation}
From the first equality, for cycles $C>0$ and $D>0$,  we have 
\begin{equation}
\label{eq:chiadd}
\chi(C+D)=\chi(C) + \chi(D) -C D.
\end{equation}

%%% Definition 2.3
\begin{defn}
 There exists a minimum $Z_f$ among the anti-nef cycles $D>0$; 
the cycle $\Z$ is called the {\em fundamental cycle} (\cite{artin.rat}).
The {\em fundamental genus} $p_f(A)$ of $A$ is defined by 
$p_f(A)=1-\chi(Z_f)$, and the {\em geometric genus} $p_g(A)$ of $A$ is defined by $p_g(A)=h^1(\cO_X)$; these are independent of the choice of the resolution and satisfy $p_f(A) \le p_g(A)$.
While the fundamental genus is determined by the resolution graph, the geometric genus is not. 
The singularity $A$ is said to be {\em rational} if $p_g(A)=0$.
It is known that $A$ is a rational singularity if and only if $p_f(A)=0$ 
(see \cite{artin.rat}).
\end{defn}

\par 
The fundamental cycle $Z_f$ can be computed 
via a sequence $\{C_1, \dots, C_m\}$ of cycles such that
\begin{equation}\label{eq:compseq}
 C_1=E_{j_1}, \quad C_i=C_{i-1}+E_{j_i} \ (1< i \le m), \quad C_m=\Z,
\end{equation}
where $E_{j_1}$ is an arbitrary component of $E$ and $C_{i-1}E_{j_i}>0$ 
for $1< i \le m$.
Such a sequence $\{C_i\}_i$ is called a  
 {\em computation sequence} for $\Z$. 
From \eqref{eq:chiadd}, we have 
\begin{equation}
\label{eq:seqchi}
\chi(C_i)=\chi(C_{i-1})+\chi(E_{j_i})- C_{i-1}E_{j_i}  \ (1< i \le m).
\end{equation}
\par 
We can verify whether $A$ is a rational singularity as follows.
%%%%%%  Proposition 2.4
\begin{prop}\label{p:ratcompseq}
Assume that each $E_i$ is a nonsingular rational curve, and let $\{C_1, \dots, C_m\}$ be a computation sequence as in \eqref{eq:compseq}.
Then, $A$ is rational if and only if $C_{i-1}E_{j_i}=1$ for $1<i\le m$.
\end{prop}
%%%%
\begin{proof}
It follows from \eqref{eq:seqchi} and the fact that $h^0(\cO_{C_i})=1$ for $1\le i \le m$ 
(see \cite[(2.6)]{La.me}).
\end{proof}

%%% Definition 2.5
\begin{defn}
Assume that $X$ is a {\em good resolution}, that is, $E$ is a simple normal crossings divisor. 
Then $E$ or its weighted dual graph is said to be {\em star-shaped} if it has at most one irreducible component $E_0$ of $E$ such that $(E-E_0)E_0\ge 3$ and any component of $E-E_0$ is a rational curve.
In this case, $E_0$ is called a {\em central curve} and each connected component of $E-E_0$ a {\em branch}.
\end{defn}

\par 
It is known that any resolution of a rational singularity is a good resolution and that if $A$ is the localization of a graded ring by the homogeneous maximal ideal and $X$ is the minimal good resolution, 
then $E$ is star-shaped.

%%%% Definition 2.6
\begin{defn}
A cycle $F \ge 0$ is called the {\em fixed part} of the divisor $L$ or the invertible
 sheaf $\cO_X(L)$ on $X$ if $F$ is the maximal cycle such that $H^0(\cO_X(L-F))=H^0(\cO_X(L))$.  
We say that $L$ {\em has no fixed components} if its fixed part is zero.
 We say that $L$ is {\em generated} if $\cO_X(L)$ is generated by $H^0(\cO_X(L))$.
Obviously, $L$ has no fixed component if $\cO_X(L)$ is generated.
\end{defn}

%%%%%%  Remark 2.7
\begin{rem} \label{r:fcomp}
From the definition, we easily see the following.
\begin{enumerate}
\item $L$ is nef if $L$ has no fixed component.
\item For $f\in H^0(\cO_X(L))\setminus\{0\}$, there exists an effective divisor $H_f$ on $X$ such that $\di_X(f)+L=H_f$. 
Then the fixed part is the greatest divisor that is less than or equal to any member of $\defset{H_f}{f\in H^0(\cO_X(L))\setminus\{0\}}$.
\item If $C>0$ is a cycle, then $\cO_X(-C)$ has no fixed components if and only if there exists an element $h\in H^0(\cO_X(-C))$ such that $\di_X(h)=C+H$, where $H$ does not contain any component of $E$.
\end{enumerate}
\end{rem}

%%%%%%%%%%%%%%%%%%%%%%%%%%%%%%%%%%%%%%
%%%%%%%%%%%%%%%%%%%%%%%%%%%%%%%%%%%%%%%%%%%%
%%%%% Section 3
\section{The canonical trace ideals of rational singularities}

\par 
In this section, for rational singularities, we provide a description of the trace of $A$-modules represented by divisors on a resolution $X$, as well as the canonical trace ideals.

\par 
We use the notation introduced in the preceding section.
Note that $A$ is rational if and only if $K_A=H^0(\cO_X(K_X))$ (cf. \cite[Theorem 3.4 and 3.5]{la.rat}) and that $X \to \spec(A)$ is the minimal resolution if and only if the canonical divisor $K_X$ is nef.
Moreover, $A$ is a rational Gorenstein singularity (or, rational double point) if and only if $K_X\equiv 0$.
\par 
Let $L$  be a divisor on $X$.

%%%%%%% Lemma 3.1
\begin{lem} \label{l:L*}
Assume that $L$ has no fixed components. 
Then 
\[
H^0(\cO_X(L))^{-1} = H^0(\cO_X(-L)).
\]
\end{lem}
%%%%
\begin{proof}
If $g\in H^0(\cO_X(-L))$, then $g H^0(\cO_X(L)) 
\subset H^0(\cO_X(-L+L)) =A$.
Therefore, $H^0(\cO_X(-L)) \subset H^0(\cO_X(L))^{-1}$.
\par 
Let us show that $H^0(\cO_X(L))^{-1} \subset H^0(\cO_X(-L))$.
Assume that $g\in H^0(\cO_X(L))^{-1}$.
For any $f\in H^0(\cO_X(L))\setminus\{0\}$, let $H_f$ denote the divisor as in \remref{r:fcomp} (2).
Since $g H^0(\cO_X(L)) \subset A$, we have 
\[
\di_X(g) + \di_X(f) = \di_X(g) - L +H_f \ge 0 \ \text{ for }
\ 
\forall f\in H^0(\cO_X(L))\setminus\{0\}.
\]
Since $L$ has no fixed components, moving $f$, we have the inequality $\di_X(g) - L \ge 0$. 
Hence we obtain that $g\in H^0(\cO_X(-L))$.
\end{proof}

\par 
The good properties of rational singularities that we use in this section are the following.

%%%%%%%%%  Lemma 3.2
\begin{lem}[cf. {\cite{artin.rat}, \cite[4.17]{chap}}]
\label{l:invrat}
Assume that $A$ is a rational singularity.  
Then, we have the following.
\begin{enumerate}
\item $\m\cO_X=\cO_X(-\Z)$.
\item $L$ is nef if and only if it is generated$;$ if this is the case, $H^1(\cO_X(L))=0$.
\end{enumerate}
\end{lem}

%%%%%%%%%  Lemma 3.3
\begin{lem}
\label{l:2ndfund}
Assume that $A$ is a rational singularity, and let $L_1$ and $L_2$ be nef divisors on $X$.
 Then $H^0( \cO_X(L_2))\cdot  H^0(\cO_X(L_1)) = H^0( \cO_X(L_1+L_2)) $.
\end{lem}
\begin{proof}
Take general elements $f_1 \in H^0(\cO_X(L_1))$ and 
$f_2\in H^0(\cO_X(L_2))$.
By \lemref{l:invrat} (2), $L_1$ and $L_2$ are generated. 
Thus we have the exact sequence
\[
0 \to \cO_X \xrightarrow {(f_1,f_2)}
 \cO_X(L_1) \oplus
\cO_X(L_2)
\xrightarrow{\genfrac{(}{)}{0pt}{}{-f_2}{f_1}}
\cO_X(L_1+L_2) \to 0.
\]
Since $h^1(\cO_X)=0$, 
 we obtain
\[
 H^0( \cO_X(L_1+L_2)) = f_1 H^0( \cO_X(L_2)) + f_2 H^0(\cO_X(L_1)).
\]
Hence the assertion follows.
\end{proof}

Now we can compute the trace ideal of a nef divisor. 
%%%%%   Theorem 3.4
\begin{thm}\label{p:rat-tr}
Assume that $A$ is a rational singularity and $L$ is nef.
Let $F$ be the fixed part of $-L$.
Then $-F$ is nef and $\tr H^0(\cO_X(L)) = H^0(\cO_X(-F))$.
Moreover, if $L\not\equiv 0$, $F$ can be computed by a computation sequence $\{C_1=\Z, \dots, C_m=F\}$ of cycles such that 
\[
C_i=C_{i-1}+E_{j_i}, \ (L+C_{i-1})E_{j_i} >0 \ \ (1<i\le m).
\]
\end{thm}
\begin{proof}
Note that $H^0(\cO_X(-L))=H^0(\cO_X(-L-F))$.
Since $L$ and $-L-F$ are nef, $-F$ is nef and the equality follows from \cite[1.1]{HHS}, \lemref{l:L*} and \lemref{l:2ndfund}.
Assume that $L\not\equiv 0$. Then $F\ne 0$.
Since $F$ is anti-nef, we have $F\ge \Z$. 
Hence a sequence as in the assertion exists because  $F$ is the  minimal cycle such that $L+F$ is anti-nef.                       
\end{proof}

Applying this result to the case $L=K_X$, we get a geometric expression of the canonical trace 
ideal of $A$.
%%%%  Theorem  3.5
\begin{thm} \label{rat-cantr}
Assume that $A$ is a rational singularity, and let $K_A$ be a canonical module of $A$. 
Then $\Tr_A(K_A)$ is an integrally closed $\m$-primary ideal which is represented on the minimal resolution $X$ as follows: 
\[
\Tr_A(K_A)=H^0(X,\mathcal{O}_X(-F)),  
\]
where $F$ is the minimal cycle such that 
$F+K_X$ is anti-nef.
Note that $F \ge \Z$ if $K_X \not \equiv 0$ since $F$ is anti-nef. 
\end{thm}
%%%%%%%%%%%%%%%%%%%%%%%%%%%%%%%%%%%%%%%%%%%%%%%%%%
%%%%%%%%%%%%%%%%%%%%%%%%%%%%%%%%%%%%%%%%%%%%%%%%%%%
%%%%%%%%%%%%%%%%%%%%%%%%%%%%%%%%%%%%%%%%%%%%%%%%%%%%%
%%%%%%%%   Section 4
\section{Criteria to be nearly Gorenstein}

In this section, we give a criterion for rational surface singularities 
to be nearly Gorenstein. 
In the following two sections, by using this criterion, we will 
provide several examples of nearly Gorenstein rational singularities, as well as a classification of nearly Gorenstein  singularities within the classes of
 almost reduced  (Definition \ref{AR-defn})
 rational singularities and quotient singularities.
\par \vspace{2mm}

%%%%%%%% Notation 4.1
\begin{nota}\label{n:F*}
Assume that $A$ is a rational singularity and $X \to \spec(A)$ is the minimal resolution. We use the following notation. 
\begin{enumerate}
\item Let $F$ denote the fixed part of $-K_X$ and 
let $\Z$ be the fundamental cycle.
\item For a cycle $C=\sum_{i=1}^nc_iE_i$, let $\cff_{E_i}(C)=c_i$.
\item Let $E_j^*$ denote the element of $\sum_{i=1}^n \Q E_i$ that satisfies $E_j^* E_i = -\delta_{ij}$.
\end{enumerate}
Note that $F$ is anti-nef by \proref{p:rat-tr} and that for any $\Q$-divisor $L$ on $X$, we have $L\equiv \sum_{i=1}^n(-LE_i)E_i^*$.
\end{nota}

\par 
We provide several characterizations of nearly Gorenstein rational singularities in terms of combinatorial conditions. 

%%%%%%%% Theorem 4.2
\begin{thm}\label{t:nGrat}
Assume that $A$ is a rational singularity, but not Gorenstein.
Then the following conditions are equivalent.
\begin{enumerate}
\item $A$ is nearly Gorenstein.
\item $\Z=F$.
\item $K_X+\Z$ is anti-nef. 
\item $\Z$ satisfies one of the following:
 \begin{enumerate}
 \item $E$ is irreducible.
 \item There exists $E_{i_0}\subset E$ such that $\cff_{E_{i_0}}(\Z)=2$ and  
\[
(\Z - E_{i_0})E_{i_0} = 1, \ \ 
(\Z - E_{i})E_{i} = 2 \ (i\ne i_0)
\]
 \item There exist $E_{i_1}, E_{i_2} \subset E$ such that 
 $\cff_{E_{i_1}}(\Z)=\cff_{E_{i_2}}(\Z)=1$ and  
\[
(\Z - E_{i})E_{i} = 1  \ (i\in \{i_1, i_2\}), \ \ 
(\Z - E_{i})E_{i} = 2 \ (i\not\in \{i_1, i_2\})
\]
 \end{enumerate}
\item  For every component $E_i$ with $E_i^2\le -3$, we have $\Z E_i \le E_i^2 + 2$.
\end{enumerate}
\end{thm}

\begin{proof}
(1) $\Leftrightarrow$ (2).
Since $K_X$ is nef and $H^0(\cO_X(K_X))=K_A$, it follows from \proref{p:rat-tr} that $\tr K_A = H^0(\cO_X(-F))$.
By \lemref{l:invrat} (1), we have $\m=H^0(\cO_X(-\Z))$, and thus $A$ is nearly Gorenstein if and only if $H^0(\cO_X(-F))=H^0(\cO_X(-\Z))$.
Since $-F$ and $-Z_f$  have no fixed components by \lemref{l:invrat} (2), 
we obtain that $H^0(\cO_X(-F))=H^0(\cO_X(-\Z))$ if and only if $F=\Z$ (cf. \lemref{r:fcomp}).

(2) $\Rightarrow$ (3).
It follows from the fact that $-K_X-F$ is nef.

(3) $\Rightarrow$ (2).
Since $K_X\not\equiv 0$,  $F$ is a nonzero anti-nef cycle and thus $F\ge \Z$.
Hence 
\[
H^0(\cO_X(-K_X-F)) = H^0(\cO_X(-K_X-\Z)) = H^0(\cO_X(-K_X)).
\]
However, since $-K_X-\Z$ has no fixed components by the assumption and \lemref{l:invrat} (2), we have $F=\Z$.

(3) $\Rightarrow$ (4).
From the equality $\Z(K_X+\Z)=-2\chi(\Z)=-2$, we easily see that one of the following holds:  
% Yos2607
 \begin{enumerate}
 \item[(a*)] There exists $E_{i_0}\subset E$ such that $\cff_{E_{i_0}}(\Z)=1$ and  $K_X+\Z \equiv 2 E_{i_0}^*$.
 \item[(b*)] There exists $E_{i_0}\subset E$ such that $\cff_{E_{i_0}}(\Z)=2$ and  $K_X+\Z \equiv  E_{i_0}^*$.
 \item[(c*)] There exist $E_{i_1}, E_{i_2} \subset E$ such that $\cff_{E_{i_1}}(\Z)=\cff_{E_{i_2}}(\Z)=1$ and  $K_X+\Z \equiv E_{i_1}^* + E_{i_2}^*$.
 \end{enumerate}
For any $i$, since $K_XE_i=-E_i^2-2$, we have
\begin{equation}
\label{eq:Z-Ei}
(K_X+\Z)E_i = -E_i^2+\Z E_i-2 = (\Z-E_i)E_i-2.
\end{equation}
It is now obvious that (b*) and (c*) are equivalent to (b) and (c), respectively.
Assume that (a*) holds.
By \eqref{eq:Z-Ei}, we have $(\Z-E_{i_0})E_{i_0} =0$.
Hence $E=E_{i_0}$.
\par 
The implications ``(4) $\Rightarrow$ (3)'' and ``(3) $\Leftrightarrow$ (5)''  immediately follow from \eqref{eq:Z-Ei}.
\end{proof}

\par 
The easiest case of nearly Gorenstein singularities is the following. 

%%%%%%%% Proposition 4.3
\begin{prop} \label{p:redchain}
If $\Z$ satisfies either $(4.a)$ or $(4.c)$ of $\thmref{t:nGrat}$, then the resolution graph is a chain, that is, $A$ is a cyclic quotient singularity. 
\end{prop}
\begin{proof}
Assume that $E$ is not irreducible and $\Z$ satisfies $(4.c)$ of \thmref{t:nGrat}. By assumption of $E_{i_1}$, $E_{i_1}$ is an end curve and there exists a unique curve $E_1$ such that $E_1E_{i_1}=1$ 
with $\cff_{E_1}(\Z)=1$.  
Since $(\Z-E_1)E_1=2$ and $\cff_{E_1}(\Z)=1$, we can choose a curve 
$E_2 \ne E_{i_1}$ such that $E_1E_2=1$ with $\cff_{E_2}(\Z)=1$. 
By repeating this procedure, we conclude that the resolution graph is 
a chain and $\Z$ is reduced. 

\begin{figure}[htb]
\begin{picture}(200,40)(-30,0)
    \thicklines
%\put(-30,20){\text{Type $A$}}
\put(25,28){{\tiny $1$}}
\put(30,22){\circle*{6}}
\put(33,22){\line(1,0){30}}
\put(24,10){{\tiny $E_{i_1}$}}
\put(55,28){{\tiny $1$}}
\put(61,22){\circle*{6}}
\put(54,10){{\tiny $E_1$}}
\put(65,22){\line(1,0){20}}
\put(80,28){{\tiny $1$}}
\put(86,22){\circle*{6}}
\put(103,19){$\cdots$}
\put(79,10){{\tiny $E_2$}}
\put(90,22){\line(1,0){10}}
\put(140,28){{\tiny $1$}}
\put(120,22){\line(1,0){21}}
\put(145,22){\circle*{6}}
%\put(116,10){{\tiny $-b_2$}}
%
\put(167,28){{\tiny $1$}}
\put(148,22){\line(1,0){23}}
\put(175,22){\circle*{6}}
\put(166,10){{\tiny $E_{i_2}$}}
\end{picture}
\caption{\label{fig:Center-nG} The resolution graph of a cyclic quotient singularity}
\end{figure}
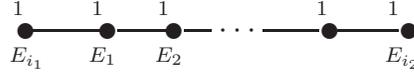
\par 
This implies that $A$ is a cyclic quotient singularity. 
\end{proof}

%%%%%%%%%% Proposition 4.4
\begin{prop}\label{p:r-1r}
Assume that $\Z$ satisfies $(4.b)$ of Theorem $\ref{t:nGrat}$ and that $\{E_1, \dots, E_r\}$ is the set of components of $E-E_{i_0}$ intersecting $E_{i_0}$.
Moreover, assume that one of following conditions holds$:$
\begin{enumerate}
\item $\cff_{E_j}(\Z)=1$ for $1\le j \le r$.
\item $E$ is star-shaped and $E_{i_0}$ is the center.
\end{enumerate}
Then, we obtain that 
\[
E=E_{i_0}+E_1+\cdots+E_r, 
\
E_{i_0}^2=-r+1.
\]
\end{prop}
%%%%%%%%%%%%%%%% Proof 
\begin{proof}
Assume that the condition (1) holds.
Then $(Z_f - E_{i_0})E_{i_0} = E E_{i_0}$.
Since  $(\Z-E_{i_0})E_{i_0}=1$, we have $E_{i_0}^2+r=1$.
For $1\le j \le r$, it follows from  $(\Z-E_j)E_j=2$ that $E_j$ intersects only $E_{i_0} \le E-E_j$.

Next, assume (2).
First, note that $2=\cff_{E_{i_0}}(\Z) \ge \cff_{E_j}(\Z)$ for $1\le j \le r$.
Since $E$ appears in a computation sequence for $Z_f$, 
$Z_f \ge E+E_{i_0}$,  and $E E_i \le 0$ for $i \ne i_0$, we have $E E_{i_0} = 1$,  
and thus $E_{i_0}^2 + r= 1$. 
On the other hand, we have
\[
E_{i_0}^2 + 1 = \Z E_{i_0} = 2 E_{i_0}^2 + \sum_{j=1}^r \cff_{E_j}(\Z).
\]
Hence $\sum_{j=1}^r \cff_{E_j}(\Z) = -E_{i_0}^2 + 1 = r$.
Therefore, (2) is reduced to the case (1).
\end{proof}

%%%%%%%%   Remark 4.5
\begin{rem}
Let $C$ be a nonsingular curve and $D$ a $\Q$-divisor on $C$ with $\deg D>0$.
Then we have a normal graded ring (cf. \cite{Wt}, \cite{TW})
\[
R(C,D):=\bigoplus_{n\ge 0}H^0(C, \cO_C(n D))T^n\subset k(C)[T]
\] 
Assume $C=\PP^1$.
Let  $P_1, \dots, P_r \in \PP^1$ be distinct points and let 
$Q\in \PP^1$ be an arbitrary point.
Let $D=(r-1)Q + \sum _{j=1}^r (1/q_j)P_j$, where $q_i$ are integers 
greater than $1$. 
If $A$ is the localization of  $R(C,D)$ with respect to 
the maximal ideal $\bigoplus_{n\ge 1}H^0(C, \cO_C(nD))T^n$,  
then it satisfies the conditions in \proref{p:r-1r}.  
\end{rem}

%%%%%%% Example 4.6
\begin{ex} \label{NonLT}
Recall that  $E=\bigcup_{i=1}^n E_i$.
Assume that $n\ge 4$, $E_1^2=-n+2$, and that for $j \ge 2$,  each $E_j$ is a connected component of $E-E_1$ with $E_j^2=-2$. 
Then $\Z=E+E_1$ and $K_X+\Z\equiv E_1^*$, that is, the condition (4.b) in \thmref{t:nGrat} is satisfied.
If $n\ge 5$, then this singularity $A$ is not log-terminal and $e(A)=n-2$. 

\begin{figure}[htb]
\begin{picture}(300,55)(0,0)
\thicklines
%\put(0,40){\text{(2-A)}}
%
\put(70,23){{\tiny $1$}}
\put(75,17){\circle{6}}
\put(70,5){{\tiny $E_2$}}
\put(93,20){{\tiny $2$}}
\put(105,19){\line(-1,1){16}}
\put(113,27){$\ddots$}
\put(78,17){\line(1,0){24}}
\put(105,17){\circle*{7}}
\put(100,5){{\tiny $E_{1}$}}
\put(78,37){{\tiny $1$}}
\put(87,37){\circle{6}}
\put(133,23){{\tiny $1$}}
\put(108,17){\line(1,0){24}}
\put(135,17){\circle{6}}
\put(130,5){{\tiny $E_n$}}
\put(96,43){{\tiny $1$}}
\put(105,19){\line(0,1){21}}
\put(105,43){\circle{6}}
\put(110,40){{\tiny $E_4$}}
%
%\put(157,23){{\tiny $2$}}
%\put(138,17){\line(1,0){24}}
%\put(165,17){\circle{6}}
%\put(158,5){{\tiny $-3$}}
  
%
\end{picture}
\caption{\label{fig:Center-nG} A rational singularity
which satisfies (4.b)
}
\end{figure}
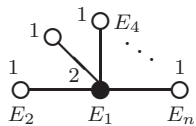
\end{ex}
%%%%%%%%%%%%%%%%%%%%%%%%%%%%%%%%%%%%%%%%%%
%%%%%%%%%%%%%%%%%%%%%%%%%%%%%%%%%%%%%%%%%%
%%%%%%%%%%%%%%%%%%%%%%%%%%%%%%%%%%%%%%%%%%
% Section 5 
\section{Rational singularities having almost reduced fundamental cycle}

\par
In general, there are many examples of resolution graphs for nearly Gorenstein rational singularities. 
But we can classify all resolution graphs of nearly Gorenstein 
rational singularities with almost reduced fundamental cycle. 
As a special case, we can recover the classification of 
nearly Gorenstein rational triple points (\cite{MY}). 
Moreover, our method will be useful to classify all nearly Gorenstein 
rational singularities of multiplicity $e$ for a given integer $e \ge 3$. 

\par 
We use the notation as in the previous section. 
%%%%%%%%% Definition 5.1
\begin{defn}[\textrm{\cite{La79}}] \label{AR-defn}
A rational singularity $(A,\m)$ has an \textit{almost reduced 
fundamental cycle} if $\cff_{E_i}(\Z)=1$ for every $E_i$ 
with $-E_i^2 \ge 3$. 
\end{defn}

%%%%% Theorem 5.2
\begin{thm} \label{ARnG}
Assume that the fundamental cycle $\Z$ is almost reduced. 
Then the following conditions are equivalent$:$ 
\begin{enumerate}
\item $A$ is nearly Gorenstein. 
\item 
The resolution graph of $X$ and $\Z$ is one of the list in 
Figure $\ref{fig:AlmostReduced}$, 
where $\circ$  stands for a $(-2)$-curve. 

\par 
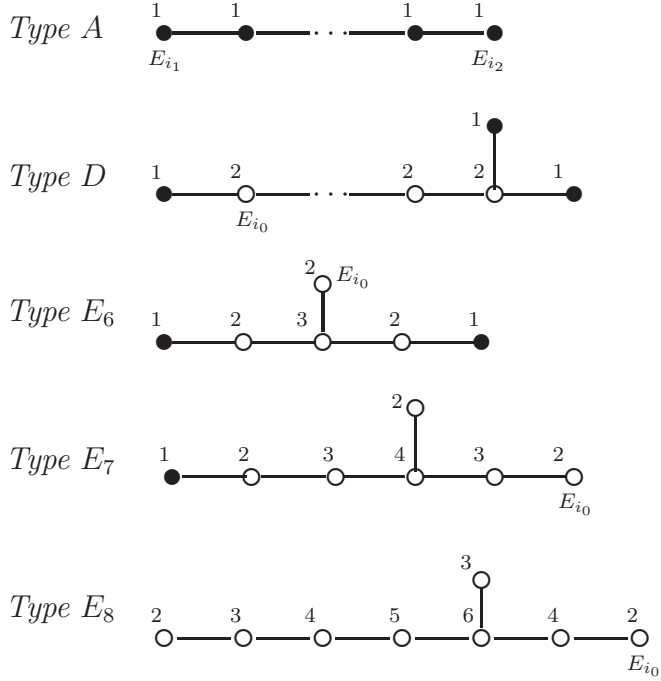
\begin{figure}[htb] 
\begin{picture}(400,60)(-30,0)
    \thicklines
\put(-30,20){\text{Type $A$}}
\put(25,28){{\tiny $1$}}
\put(30,22){\circle*{6}}
\put(33,22){\line(1,0){30}}
\put(24,10){{\tiny $E_{i_1}$}}
\put(55,28){{\tiny $1$}}
\put(61,22){\circle*{6}}
%\put(50,10){{\tiny $-b_{n-1}$}}
%
\put(65,22){\line(1,0){20}}
\put(86,19){$\cdots$}
\put(120,28){{\tiny $1$}}
\put(100,22){\line(1,0){21}}
\put(125,22){\circle*{6}}
%\put(116,10){{\tiny $-b_2$}}
%
\put(147,28){{\tiny $1$}}
\put(128,22){\line(1,0){23}}
\put(155,22){\circle*{6}}
\put(146,10){{\tiny $E_{i_2}$}}
\end{picture}

\begin{picture}(400,55)(-30,0)
    \thicklines
\put(-30,20){\text{Type $D$}}
\put(25,23){{\tiny $1$}}
\put(30,17){\circle*{6}}
\put(33,17){\line(1,0){25}}
%\put(24,5){{\tiny $-b_n$}}
%
\put(55,23){{\tiny $2$}}
\put(61,17){\circle{6}}
%\put(50,5){{\tiny $-b_{n-1}$}}
  \put(57,5){{\tiny $E_{i_0}$}}
\put(65,17){\line(1,0){20}}
\put(86,14){$\cdots$}
\put(120,23){{\tiny $2$}}
\put(100,17){\line(1,0){21}}
\put(125,17){\circle{6}}
%\put(116,5){{\tiny $-b_1$}}
%
\put(147,23){{\tiny $2$}}
\put(128,17){\line(1,0){23}}
\put(155,17){\circle{6}}
%\put(148,5){{\tiny $-b$}}
%
\put(177,23){{\tiny $1$}}
\put(158,17){\line(1,0){24}}
\put(185,17){\circle*{6}}
\put(146,43){{\tiny $1$}}
\put(155,19){\line(0,1){21}}
\put(155,43){\circle*{6}}
\end{picture}

\begin{picture}(400,50)(-30,0)
    \thicklines
 \put(-30,20){\text{Type $E_6$}}   
  \put(25,18){{\tiny $1$}}
\put(30,12){\circle*{6}}
\put(33,12){\line(1,0){24}}
  \put(55,18){{\tiny $2$}}
\put(60,12){\circle{6}}
\put(63,12){\line(1,0){24}}
  \put(80,18){{\tiny $3$}}
\put(90,12){\circle{6}}
\put(93,12){\line(1,0){24}}
  \put(115,18){{\tiny $2$}}
\put(120,12){\circle{6}}
\put(123,12){\line(1,0){24}}
  \put(145,18){{\tiny $1$}}
\put(150,12){\circle*{6}}
\put(90,16){\line(0,1){15}}
  \put(83,38){{\tiny $2$}}
\put(90,34){\circle{6}}
  \put(95,35){{\tiny $E_{i_0}$}}
\end{picture}

\begin{picture}(400,55)(30,0)
\thicklines
 \put(30,20){\text{Type $E_7$}}   
\put(88,23){{\tiny $1$}}
\put(97,17){\line(1,0){24}}
\put(93,17){\circle*{6}}
%\put(183,5){{\tiny $$}}
%
\put(118,23){{\tiny $2$}}
\put(127,17){\line(1,0){24}}
\put(123,17){\circle{6}}
%\put(113,5){{\tiny $$}}
%
\put(150,23){{\tiny $3$}}
\put(155,17){\circle{6}}
%\put(146,5){{\tiny $-3$}}
%
\put(177,23){{\tiny $4$}}
%\put(187,23){{$E_0$}}
\put(158,17){\line(1,0){23}}
\put(185,17){\circle{6}}
%\put(178,5){{\tiny $-3$}}
%
\put(207,23){{\tiny $3$}}
\put(188,17){\line(1,0){24}}
\put(215,17){\circle{6}}
%\put(208,5){{\tiny $-3$}}
%
\put(176,43){{\tiny $2$}}
\put(185,19){\line(0,1){21}}
\put(185,43){\circle{6}}
\put(237,23){{\tiny $2$}}
\put(218,17){\line(1,0){24}}
\put(245,17){\circle{6}}
%\put(238,5){{\tiny $-3$}}
  \put(239,5){{\tiny $E_{i_0}$}}
\end{picture}

\begin{picture}(400,55)(-30,0)
    \thicklines
 \put(-30,20){\text{Type $E_8$}}   
  \put(25,18){{\tiny $2$}}
\put(30,12){\circle{6}}
\put(25,0){{\tiny $E_{i_0}$}}   
\put(35,12){\line(1,0){20}}
  \put(55,18){{\tiny $3$}}
\put(60,12){\circle{6}}
\put(65,12){\line(1,0){20}}
  \put(83,18){{\tiny $4$}}
\put(90,12){\circle{6}}
\put(95,12){\line(1,0){20}}
  \put(115,18){{\tiny $5$}}
\put(120,12){\circle{6}}
\put(125,12){\line(1,0){20}}
  \put(143,18){{\tiny $6$}}
\put(150,12){\circle{6}}
\put(155,12){\line(1,0){20}}
  \put(175,18){{\tiny $4$}}
\put(180,12){\circle{6}}
\put(185,12){\line(1,0){20}}
  \put(205,18){{\tiny $2$}}
\put(210,12){\circle{6}}
%  \put(205,0){{\tiny $E_{i_0}$}}
%
\put(150,16){\line(0,1){15}}
  \put(142,38){{\tiny $3$}}
\put(150,34){\circle{6}}
\end{picture}
\caption{\label{fig:AlmostReduced} nG rational singularity of almost reduced fundamental cycle  
}
\end{figure}
\end{enumerate}
When this is the case, the resolution graph is star-shaped. 
\end{thm}
%%%%%
\begin{proof}
$(2) \Longrightarrow (1):$ 
First note that $A$ has a resolution graph as above, then it is a rational singularity. 
If $\Z$ is of Type $A$, then it satisfies (4.a) (resp. (4.c)) 
of Theorem $\ref{t:nGrat}$ if $n=1$ (resp. 
$n \ge 2$ and $E_{i_1}$ and $E_{i_2}$ are end curves). 
Moreover, if $\Z$ is either one of Type $D$, $E_6$, $E_7$ or $E_8$, then it satisfies (4.b) of Theorem $\ref{t:nGrat}$. 
Hence $A$ is nearly Gorenstein. 
\par \vspace{2mm}
$(1) \Longrightarrow (2):$
By virtue of Theorem $\ref{t:nGrat}$ and Proposition \ref{p:redchain}, 
we may assume that $\Z$ satisfies $(4.b)$ of Theorem $\ref{t:nGrat}$. 
For each $i$, we put 
\[
z_i=\cff_{E_i}(\Z), \quad d(E_i)=\sum_{E_iE_j=1}z_j \quad  \text{and} \; 
b_i=-E_i^2. 
\]
Since $\Z$ is almost reduced and $z_{i_0}=2$, we have 
$b_{i_0}=2$. 
Thus by assumption, we get 
\[
d(E_{i_0})=(\Z-E_{i_0})E_{i_0}+(z_{i_0}-1)(-E_{i_0}^2)=1+b_{i_0}=3. 
\]
Similarly, we have 
\begin{equation} \label{eqdeg}
d(E_i)=(z_i-1)b_i+2 \;\;\text{for every $i \ne i_0$}. 
\end{equation}

\begin{flushleft}
{\bf Case 1.} The case where $E_{i_0}$ is a node 
(i.e. $\sharp\{E_j\,|\, E_{i_0}E_j=1\} \ge 3$). 
\end{flushleft}
\par
Since $d(E_{i_0})=3$, we can put $\{E_j\,|\, E_{i_0}E_j=1\}=\{E_1,E_2,E_3\}$. 
Then since $z_1+z_2+z_3=d(E_{i_0})=3$ and every $z_i \ge 1$, 
we have $z_1=z_2=z_3=1$. 
Then $\Z$ is Type D with $n=4$ (cf. Proposition \ref{p:r-1r}). 

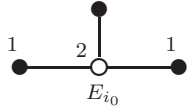
\begin{picture}(300,40)(-60,0)
    \thicklines 
%  \put(25,18){{\tiny $1$}}
%\put(30,12){\circle*{6}}
%\put(33,12){\line(1,0){24}}
  \put(55,18){{\tiny $1$}}
\put(60,12){\circle*{6}}
\put(63,12){\line(1,0){24}}
  \put(81,16){{\tiny $2$}}
\put(90,12){\circle{6}}
\put(93,12){\line(1,0){24}}
  \put(115,18){{\tiny $1$}}
\put(120,12){\circle*{6}}
%\put(123,12){\line(1,0){24}}
%  \put(145,18){{\tiny $1$}}
%\put(150,12){\circle*{6}}
\put(90,16){\line(0,1){15}}
  \put(83,38){{\tiny $1$}}
\put(90,34){\circle*{6}}
  \put(85,0){{\tiny $E_{i_0}$}}
\end{picture}

\begin{flushleft}
{\bf Case 2.} The case where $E_{i_0}$ is neither a node nor an end. 
\end{flushleft}
\par
As $\sharp\{E_j\,|\, E_{i_0}E_j=1\} =2$ and $d(E_{i_0})=3$, 
there exist two curves  $E_1$, $E_2$ such that $z_1=1$ and $z_2=2$. 
Then as $d(E_1)=2 = z_{i_0}$, $E_1$ is an end curve. 
Moreover, $z_2=2$ yields $b_2=2$. 
Then since $d(E_2)-z_{i_0}=(b_2+2)-2 = b_2=2$, either 
one of the following cases occurs (as subgraphs): 
\par
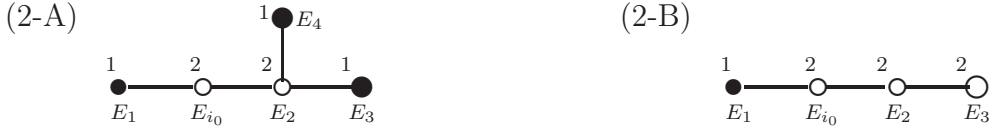
\begin{picture}(300,55)(0,0)
\thicklines
\put(0,40){\text{(2-A)}}
%\put(8,23){{\tiny $1$}}
%\put(17,17){\line(1,0){24}}
%\put(13,17){\circle*{6}}
%\put(12,5){{\tiny $E_1$}}
%
\put(38,23){{\tiny $1$}}
\put(47,17){\line(1,0){24}}
\put(43,17){\circle*{6}}
\put(40,5){{\tiny $E_{1}$}}
\put(70,23){{\tiny $2$}}
\put(75,17){\circle{6}}
\put(70,5){{\tiny $E_{i_0}$}}
\put(97,23){{\tiny $2$}}
\put(78,17){\line(1,0){23}}
\put(105,17){\circle{6}}
\put(100,5){{\tiny $E_2$}}
\put(127,23){{\tiny $1$}}
\put(108,17){\line(1,0){24}}
\put(135,17){\circle*{8}}
\put(130,5){{\tiny $E_3$}}
\put(96,43){{\tiny $1$}}
\put(105,19){\line(0,1){21}}
\put(105,43){\circle*{8}}
\put(110,40){{\tiny $E_4$}}
%
%\put(157,23){{\tiny $2$}}
%\put(138,17){\line(1,0){24}}
%\put(165,17){\circle{6}}
%\put(158,5){{\tiny $-3$}}
  
%
\end{picture}
\begin{picture}(300,55)(70,0)
\thicklines
\put(0,40){\text{(2-B)}}
%\put(8,23){{\tiny $1$}}
%\put(17,17){\line(1,0){24}}
%\put(13,17){\circle*{6}}
%\put(12,5){{\tiny $E_1$}}
%
\put(38,23){{\tiny $1$}}
\put(47,17){\line(1,0){24}}
\put(43,17){\circle*{6}}
\put(40,5){{\tiny $E_{1}$}}
\put(70,23){{\tiny $2$}}
\put(75,17){\circle{6}}
\put(70,5){{\tiny $E_{i_0}$}}
\put(97,23){{\tiny $2$}}
\put(78,17){\line(1,0){22}}
\put(105,17){\circle{6}}
\put(100,5){{\tiny $E_2$}}
\put(127,23){{\tiny $2$}}
\put(108,17){\line(1,0){24}}
\put(135,17){\circle{8}}
\put(130,5){{\tiny $E_3$}}
%
%\put(96,43){{\tiny $1$}}
%\put(105,19){\line(0,1){21}}
%\put(105,43){\circle*{8}}
%\put(110,40){{\tiny $E_4$}}
%
%\put(157,23){{\tiny $2$}}
%\put(138,17){\line(1,0){24}}
%\put(165,17){\circle{6}}
%\put(158,5){{\tiny $-3$}}
  
%
\end{picture}

\par 
Suppose (2-A). 
As $z_i=1$  for $i=3,4$, one has $d(E_i)=2=z_2$. 
Hence $E_i$ is an end curve for $i=3,4$. 
Thus this graph is Type $D$ (with $n=5$). 
\par 
Next suppose (2-B). Here $z_3=2$ yields $b_3=2$. 
By a similar argument as above, one can obtain either one of the 
following cases:
%%%
\par 
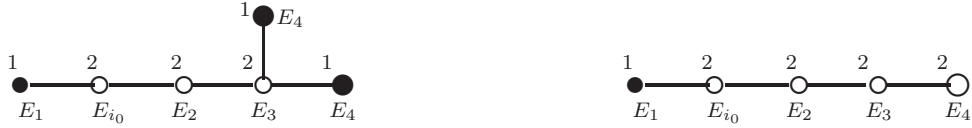
\begin{picture}(300,55)(0,0)
\thicklines
%\put(0,40){\text{(2-A)}}
\put(8,23){{\tiny $1$}}
\put(17,17){\line(1,0){24}}
\put(13,17){\circle*{6}}
\put(12,5){{\tiny $E_1$}}
\put(38,23){{\tiny $2$}}
\put(47,17){\line(1,0){24}}
\put(43,17){\circle{6}}
\put(40,5){{\tiny $E_{i_0}$}}
\put(70,23){{\tiny $2$}}
\put(75,17){\circle{6}}
\put(70,5){{\tiny $E_2$}}
\put(97,23){{\tiny $2$}}
\put(78,17){\line(1,0){23}}
\put(105,17){\circle{6}}
\put(100,5){{\tiny $E_3$}}
\put(127,23){{\tiny $1$}}
\put(108,17){\line(1,0){24}}
\put(135,17){\circle*{8}}
\put(130,5){{\tiny $E_4$}}
\put(96,43){{\tiny $1$}}
\put(105,19){\line(0,1){21}}
\put(105,43){\circle*{8}}
\put(110,40){{\tiny $E_5$}}  
%
%\put(157,23){{\tiny $2$}}
%\put(138,17){\line(1,0){24}}
%\put(165,17){\circle{6}}
%\put(158,5){{\tiny $-3$}}
  
%
\end{picture}
\begin{picture}(300,55)(70,0)
\thicklines
%\put(0,40){\text{(2-B)}}
\put(8,23){{\tiny $1$}}
\put(17,17){\line(1,0){24}}
\put(13,17){\circle*{6}}
\put(12,5){{\tiny $E_1$}}
\put(38,23){{\tiny $2$}}
\put(47,17){\line(1,0){24}}
\put(43,17){\circle{6}}
\put(40,5){{\tiny $E_{i_0}$}}
\put(70,23){{\tiny $2$}}
\put(75,17){\circle{6}}
\put(70,5){{\tiny $E_2$}}
\put(97,23){{\tiny $2$}}
\put(78,17){\line(1,0){22}}
\put(105,17){\circle{6}}
\put(100,5){{\tiny $E_3$}}
\put(127,23){{\tiny $2$}}
\put(108,17){\line(1,0){24}}
\put(135,17){\circle{8}}
\put(130,5){{\tiny $E_4$}}
%
%\put(96,43){{\tiny $1$}}
%\put(105,19){\line(0,1){21}}
%\put(105,43){\circle*{8}}
%\put(110,40){{\tiny $E_4$}}
%
%\put(157,23){{\tiny $2$}}
%\put(138,17){\line(1,0){24}}
%\put(165,17){\circle{6}}
%\put(158,5){{\tiny $-3$}}
%
\end{picture}
%%%
By repeating this procedure, we conclude that $\Z$ is type D. 
%%%
\begin{flushleft}
{\bf Case 3.} The case where $E_{i_0}$ is an end. 
\end{flushleft}
\par 
Note that $b_{i_0}=2$ because $z_{i_0}=2 > 1$. 
As $E_{i_0}$ is an end curve, there exists a curve $E_1$ with 
$z_1=3$ and hence $b_1=2$. 
We may assume $\Z$ contains the following chain such that 
$E_m$ is a node$:$ 
\par 
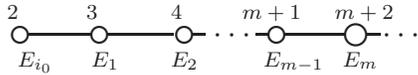
\begin{picture}(400,40)(0,0)
\thicklines
%\put(0,40){\text{(2-B)}}
\put(8,23){{\tiny $2$}}
\put(16,17){\line(1,0){24}}
\put(13,17){\circle{6}}
\put(12,5){{\tiny $E_{i_0}$}}
\put(38,23){{\tiny $3$}}
\put(46,17){\line(1,0){25}}
\put(43,17){\circle{6}}
\put(40,5){{\tiny $E_{1}$}}
\put(70,23){{\tiny $4$}}
\put(75,17){\circle{6}}
\put(70,5){{\tiny $E_2$}}
\put(97,23){{\tiny $m+1$}}
\put(78,17){\line(1,0){5}}
\put(86,14){$\cdots$} 
\put(102,17){\line(1,0){5}}
\put(110,17){\circle{6}}
\put(105,5){{\tiny $E_{m-1}$}}
\put(132,23){{\tiny $m+2$}}
\put(113,17){\line(1,0){23}}
\put(140,17){\circle{8}}
\put(135,5){{\tiny $E_m$}}
\put(144,17){\line(1,0){5}}
\put(150,14){$\cdots$} 
%
%\put(162,23){{\tiny $2$}}
%\put(170,17){\circle{6}}
%\put(163,5){{\tiny $-3$}}
%
\end{picture}
\par \noindent
where $r:=\sharp\{E_j \,|\, E_jE_m=1\} -1 \ge 2$. 
%%%
\begin{flushleft}
{\bf Case 3-A.} The case where $m=2k-1$, $k \ge 1$. 
\end{flushleft}
\par 
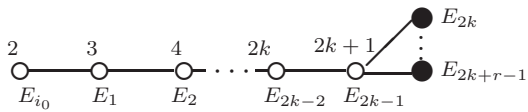
\begin{picture}(400,40)(0,0)
\thicklines
%\put(0,40){\text{(2-B)}}
\put(8,23){{\tiny $2$}}
\put(16,17){\line(1,0){24}}
\put(13,17){\circle{6}}
\put(12,5){{\tiny $E_{i_0}$}}
\put(38,23){{\tiny $3$}}
\put(46,17){\line(1,0){25}}
\put(43,17){\circle{6}}
\put(40,5){{\tiny $E_1$}}
\put(70,23){{\tiny $4$}}
\put(75,17){\circle{6}}
\put(70,5){{\tiny $E_2$}}
\put(99,23){{\tiny $2k$}}
\put(78,17){\line(1,0){5}}
\put(86,14){$\cdots$} 
\put(102,17){\line(1,0){5}}
\put(110,17){\circle{6}}
\put(105,5){{\tiny $E_{2k-2}$}}
\put(124,24){{\tiny $2k+1$}}
\put(113,17){\line(1,0){23}}
\put(140,17){\circle{6}}
\put(135,5){{\tiny $E_{2k-1}$}}
\put(143,16){\line(1,0){18}}
\put(144,19){\line(1,1){17}}
\put(165,17){\circle*{8}}
\put(163,21){$\vdots$}
\put(165,37){\circle*{8}}
\put(172,15){{\tiny $E_{2k+r-1}$}}
\put(172,35){{\tiny $E_{2k}$}}
%
%\put(162,23){{\tiny $2$}}
%\put(170,17){\circle{6}}
%\put(163,5){{\tiny $-3$}}
%
\end{picture}
\par \vspace{1mm}
Put $\{E_j\,|\, E_jE_{2k-1}=1 \}=\{E_{2k-2},E_{2k},\ldots,E_{2k+r-1}\}$
with $z_{2k} \le z_{2k+1} \le \cdots \le z_{2k+r-1}$. 
If $z_{2k}=1$, then $d(E_{2k})=2<2k+1=z_{2k-1}$. 
This is a contradiction. Hence $z_{2k} \ge 2$, $z_{2k+1},\ldots,z_{2k+r-1} \ge 2$ 
and $b_{2k}=\cdots=b_{2k+r-1}=2$. 
It follows from $\Z E_{2k} \le 0$ that 
$2z_{2k} \ge z_{2k-1}=2k+1$. 
Thus $z_{2k} \ge \lceil \frac{2k+1}{2} \rceil=k+1$. 
On the other hand, since 
\[
z_{2k}+z_{2k+1}+\cdots+z_{2k+r-1} =d(E_{2k-1})-z_{2k-2}=
2k \cdot b_{2k-1}+2-2k=2k+2,
\]
we get $r=2$ and $z_{2k}=z_{2k+1}=k+1$. 
In particular, $b_{2k}=b_{2k+1}=2$. 
\par 
\begin{picture}(400,50)(0,0)
\thicklines
%\put(0,40){\text{(2-B)}}

%\put(70,23){{\tiny $4$}}
%\put(75,17){\circle{6}}
%\put(70,5){{\tiny $E_2$}}
%
\put(99,23){{\tiny $2k$}}
%\put(78,17){\line(1,0){5}}
\put(86,14){$\cdots$} 
\put(102,17){\line(1,0){5}}
\put(110,17){\circle{6}}
\put(102,5){{\tiny $E_{2k-2}$}}
\put(124,24){{\tiny $2k+1$}}
\put(113,17){\line(1,0){23}}
\put(140,17){\circle{6}}
\put(134,5){{\tiny $E_{2k-1}$}}
\put(143,16){\line(1,0){18}}
\put(144,19){\line(1,1){17}}
\put(165,17){\circle{6}}
\put(165,37){\circle{6}}
\put(158,22){{\tiny $k+1$}}
\put(158,42){{\tiny $k+1$}}
\put(171,10){{\tiny $E_{2k+1}$}}
\put(171,34){{\tiny $E_{2k}$}}
%
%\put(162,23){{\tiny $2$}}
%\put(170,17){\circle{6}}
%\put(163,5){{\tiny $-3$}}
%
\end{picture}
\par \vspace{1mm}
As $d(E_{2k})-z_{2k-1}=k b_{2k}+2-(2k+1)=1$, 
we can find a unique curve 
$E_{2k+2}$ such that $E_{2k}E_{2k+2}=1$ with $z_{2k+2}=1$.  
% T Yos2607
\par 
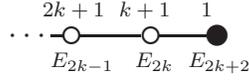
\begin{picture}(400,40)(0,0)
\thicklines
\put(99,24){{\tiny $2k+1$}}
%\put(78,17){\line(1,0){5}}
\put(86,14){$\cdots$} 
\put(102,17){\line(1,0){5}}
\put(110,17){\circle{6}}
\put(102,5){{\tiny $E_{2k-1}$}}
\put(128,24){{\tiny $k+1$}}
\put(113,17){\line(1,0){23}}
\put(140,17){\circle{6}}
\put(134,5){{\tiny $E_{2k}$}}
\put(143,17){\line(1,0){18}}
\put(165,17){\circle*{8}}
\put(159,24){{\tiny $1$}}
\put(154,5){{\tiny $E_{2k+2}$}}
\end{picture}
\par \vspace{1mm}
Then $d(E_{2k+2})=2 \ge z_{2k}=k+1$ yields $k=1$. 
Hence we conclude that $\Z$ is Type $E_6$. 
%%%%%%%
\begin{flushleft}
{\bf Case 3-B.} The case where $m=2k$, $k \ge 1$. 
\end{flushleft}
\par 
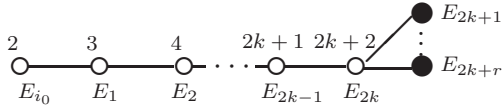
\begin{picture}(400,45)(0,0)
\thicklines
\put(8,23){{\tiny $2$}}
\put(16,17){\line(1,0){24}}
\put(13,17){\circle{6}}
\put(12,5){{\tiny $E_{i_0}$}}
\put(38,23){{\tiny $3$}}
\put(46,17){\line(1,0){25}}
\put(43,17){\circle{6}}
\put(40,5){{\tiny $E_1$}}
\put(70,23){{\tiny $4$}}
\put(75,17){\circle{6}}
\put(70,5){{\tiny $E_2$}}
\put(97,24){{\tiny $2k+1$}}
\put(78,17){\line(1,0){5}}
\put(86,14){$\cdots$} 
\put(102,17){\line(1,0){5}}
\put(110,17){\circle{6}}
\put(104,5){{\tiny $E_{2k-1}$}}
\put(124,24){{\tiny $2k+2$}}
\put(113,17){\line(1,0){23}}
\put(140,17){\circle{6}}
\put(135,5){{\tiny $E_{2k}$}}
\put(143,16){\line(1,0){18}}
\put(144,19){\line(1,1){17}}
\put(165,17){\circle*{8}}
\put(163,21){$\vdots$}
\put(165,37){\circle*{8}}
\put(172,15){{\tiny $E_{2k+r}$}}
\put(172,35){{\tiny $E_{2k+1}$}}
%
%\put(162,23){{\tiny $2$}}
%\put(170,17){\circle{6}}
%\put(163,5){{\tiny $-3$}}
%
\end{picture}
\par \vspace{1mm}
Put $\{E_j\,|\, E_jE_{2k}=1 \}=\{E_{2k-1},E_{2k+1},\cdots,E_{2k+r}\}$, where $r \ge 2$ and $z_{2k+1} \le z_{2k+2} \le \ldots \le z_{2k+r}$. 
If $z_{2k+1}=1$, then $d(E_{2k+1})=2< 2k+2=z_{2k}$. 
This is a contradiction. 
Hence $z_{2k+1} \ge 2$ and thus $b_{2k+1}=2$. 
Since $\Z E_{2k+1} \le 0$, we have $2 \cdot z_{2k+1} \ge z_{2k}=2k+2$. 
Thus $z_{2k+1} \ge k+1$. 
On the other hand, since 
\[
z_{2k+1}+z_{2k+2}+\cdots+z_{2k+r} =d(E_{2k})-z_{2k-1}=
(2k+1) \cdot b_{2k}+2-(2k+1)=2k+3,
\]
we get $r=2$ and $z_{2k+1}=k+1$ and $z_{2k+2}=k+2$. 
In particular, $b_{2k+1}=b_{2k+2}=2$. 
\par 
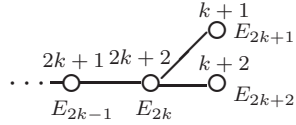
\begin{picture}(400,50)(0,0)
\thicklines
%\put(0,40){\text{(2-B)}}

%\put(70,23){{\tiny $4$}}
%\put(75,17){\circle{6}}
%\put(70,5){{\tiny $E_2$}}
%
\put(99,23){{\tiny $2k+1$}}
%\put(78,17){\line(1,0){5}}
\put(86,14){$\cdots$} 
\put(102,17){\line(1,0){5}}
\put(110,17){\circle{6}}
\put(102,5){{\tiny $E_{2k-1}$}}
\put(124,24){{\tiny $2k+2$}}
\put(113,17){\line(1,0){23}}
\put(140,17){\circle{6}}
\put(134,5){{\tiny $E_{2k}$}}
\put(143,16){\line(1,0){18}}
\put(144,19){\line(1,1){17}}
\put(165,17){\circle{6}}
\put(165,37){\circle{6}}
\put(158,22){{\tiny $k+2$}}
\put(158,42){{\tiny $k+1$}}
\put(171,10){{\tiny $E_{2k+2}$}}
\put(171,34){{\tiny $E_{2k+1}$}}
%
%\put(162,23){{\tiny $2$}}
%\put(170,17){\circle{6}}
%\put(163,5){{\tiny $-3$}}
%
\end{picture}
\par \vspace{1mm}
As $d(E_{2k+1})=2k+2=z_{2k}$, $E_{2k+1}$ is an end curve. 
Since $d(E_{2k+2})=2(k+1)+2=z_{2k}+2$, one of the following two cases 
occurs:
\par
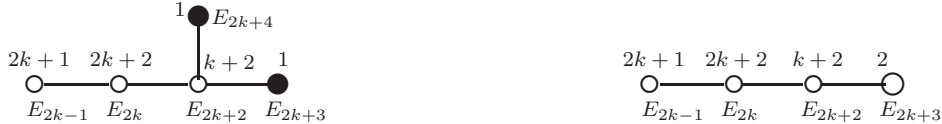
\begin{picture}(300,55)(0,0)
\thicklines
%\put(8,23){{\tiny $1$}}
%\put(17,17){\line(1,0){24}}
%\put(13,17){\circle*{6}}
%\put(12,5){{\tiny $E_1$}}
%
\put(33,24){{\tiny $2k+1$}}
\put(47,17){\line(1,0){24}}
\put(43,17){\circle{6}}
\put(40,5){{\tiny $E_{2k-1}$}}
\put(64,24){{\tiny $2k+2$}}
\put(75,17){\circle{6}}
\put(70,5){{\tiny $E_{2k}$}}
\put(107,23){{\tiny $k+2$}}
\put(78,17){\line(1,0){23}}
\put(105,17){\circle{6}}
\put(100,5){{\tiny $E_{2k+2}$}}
\put(135,24){{\tiny $1$}}
\put(108,17){\line(1,0){24}}
\put(135,17){\circle*{8}}
\put(130,5){{\tiny $E_{2k+3}$}}
\put(96,43){{\tiny $1$}}
\put(105,19){\line(0,1){21}}
\put(105,43){\circle*{8}}
\put(110,40){{\tiny $E_{2k+4}$}}
%
%\put(157,23){{\tiny $2$}}
%\put(138,17){\line(1,0){24}}
%\put(165,17){\circle{6}}
%\put(158,5){{\tiny $-3$}}
  
%
\end{picture}
\begin{picture}(300,55)(70,0)
\thicklines
%\put(8,23){{\tiny $1$}}
%\put(17,17){\line(1,0){24}}
%\put(13,17){\circle*{6}}
%\put(12,5){{\tiny $E_1$}}
%
\put(33,24){{\tiny $2k+1$}}
\put(47,17){\line(1,0){24}}
\put(43,17){\circle{6}}
\put(40,5){{\tiny $E_{2k-1}$}}
\put(64,24){{\tiny $2k+2$}}
\put(75,17){\circle{6}}
\put(70,5){{\tiny $E_{2k}$}}
\put(97,24){{\tiny $k+2$}}
\put(78,17){\line(1,0){23}}
\put(105,17){\circle{6}}
\put(100,5){{\tiny $E_{2k+2}$}}
\put(129,24){{\tiny $2$}}
\put(108,17){\line(1,0){24}}
\put(135,17){\circle{8}}
\put(130,5){{\tiny $E_{2k+3}$}}
%
%\put(96,43){{\tiny $1$}}
%\put(105,19){\line(0,1){21}}
%\put(105,43){\circle*{8}}
%\put(110,40){{\tiny $E_4$}}
%
%\put(157,23){{\tiny $2$}}
%\put(138,17){\line(1,0){24}}
%\put(165,17){\circle{6}}
%\put(158,5){{\tiny $-3$}}
%
\end{picture}
\par \vspace{1mm}
If the former case occurs, then $2=d(E_{2k+3}) \ge z_{2k+2}=k+2$.
This is a contradiction. 
Hence we may assume that the latter case occurs. 
Then 
\[
k+2=z_{2k+2} \le d(E_{2k+3}) = b_{2k+3}+2=4.  
\]
Hence $k=1,2$. 

\par 
First suppose that $k=1$. 
\par \vspace{2mm}
\begin{picture}(300,55)(0,0)
\thicklines
%\put(0,40){\text{(2-A)}}
\put(8,23){{\tiny $2$}}
\put(17,17){\line(1,0){24}}
\put(13,17){\circle{6}}
\put(12,5){{\tiny $E_{i_0}$}}
\put(38,23){{\tiny $3$}}
\put(47,17){\line(1,0){24}}
\put(43,17){\circle{6}}
\put(40,5){{\tiny $E_{1}$}}
\put(66,23){{\tiny $4$}}
\put(75,17){\circle{6}}
\put(70,5){{\tiny $E_2$}}
\put(97,23){{\tiny $3$}}
\put(78,17){\line(1,0){23}}
\put(105,17){\circle{6}}
\put(100,5){{\tiny $E_4$}}
\put(127,23){{\tiny $2$}}
\put(108,17){\line(1,0){24}}
\put(135,17){\circle{6}}
\put(130,5){{\tiny $E_5$}}
\put(157,23){{\tiny $1$}}
\put(138,17){\line(1,0){24}}
\put(165,17){\circle*{7}}
\put(160,5){{\tiny $E_6$}}
\put(66,43){{\tiny $1$}}
\put(75,19){\line(0,1){21}}
\put(75,43){\circle{6}}
\put(80,40){{\tiny $E_3$}}  
%
%\put(157,23){{\tiny $2$}}
%\put(138,17){\line(1,0){24}}
%\put(165,17){\circle{6}}
%\put(158,5){{\tiny $-3$}}
  
%
\end{picture}
\par 
Since $d(E_5)=(2-1)b_5+2=4$ and $z_4=3$, there exists a curve 
$E_6$ such that $z_6=1$. 
Then $\Z$ is Type $(E_7)$. 

\par \vspace{2mm}
Next suppose that $k=2$. 
\par \vspace{1mm}
\begin{picture}(300,55)(0,0)
\thicklines
%\put(0,40){\text{(2-A)}}
\put(8,23){{\tiny $2$}}
\put(17,17){\line(1,0){24}}
\put(13,17){\circle{6}}
\put(12,5){{\tiny $E_{i_0}$}}
\put(38,23){{\tiny $3$}}
\put(47,17){\line(1,0){24}}
\put(43,17){\circle{6}}
\put(40,5){{\tiny $E_{1}$}}
\put(70,23){{\tiny $4$}}
\put(75,17){\circle{6}}
\put(70,5){{\tiny $E_2$}}
\put(97,23){{\tiny $5$}}
\put(78,17){\line(1,0){23}}
\put(105,17){\circle{6}}
\put(100,5){{\tiny $E_3$}}
\put(127,23){{\tiny $6$}}
\put(108,17){\line(1,0){24}}
\put(135,17){\circle{6}}
\put(130,5){{\tiny $E_4$}}
\put(157,23){{\tiny $4$}}
\put(138,17){\line(1,0){24}}
\put(165,17){\circle{6}}
\put(160,5){{\tiny $E_6$}}
\put(187,23){{\tiny $2$}}
\put(168,17){\line(1,0){24}}
\put(195,17){\circle{6}}
\put(190,5){{\tiny $E_7$}}
\put(126,43){{\tiny $3$}}
\put(135,19){\line(0,1){21}}
\put(135,43){\circle{6}}
\put(140,40){{\tiny $E_5$}}  

%
%\put(157,23){{\tiny $2$}}
%\put(138,17){\line(1,0){24}}
%\put(165,17){\circle{6}}
%\put(158,5){{\tiny $-3$}}
%
\end{picture}
\par \vspace{1mm}
Since $d(E_7)=(2-1)b_7+2=4=z_6$, $E_7$ is an end curve. 
Then $\Z$ is Type $(E_8)$. 
\end{proof}

%%%%%%%%%%%%%%%%%%%%%%%%%%%%%%%%%%%%%%%%%%
\par 
Nearly Gorenstein, rational triple points (rational singularities with $e(A)=3)$ are classified in \cite{MY}.    
Since it is known that the fundamental cycle of rational triple points 
is almost reduced, one can apply Theorem $\ref{ARnG}$. 

%%%%%%%  Example 5.3
\begin{ex}[See \cite{MY} for details] \label{e(A)=3} 
If $A$ is a nearly Gorenstein rational triple point, then the resolution 
 graph of $\Z$ is  one of the following$:$ 
(1) $A_{0,m,n}$ (2) $B_{0,n}$, (3) $C_{0,n}$, (4) $D_0$, (5) $F_0$, 
where $\bullet$ (resp. $\circ$) denotes a $(-3)$-curve 
(resp. a $(-2)$-curve).  
\par 
\begin{figure}[htb]
\begin{picture}(400,60)(-30,0)
    \thicklines
\put(-30,20){\text{$(A_{0,m,n})$}}
\put(25,28){{\tiny $1$}}
\put(30,22){\circle{6}}
\put(33,22){\line(1,0){20}}
\put(55,28){{\tiny $1$}}
\put(58,22){\circle{6}}
%\put(50,10){{\tiny $-b_{n-1}$}}
%
\put(62,22){\line(1,0){20}}
\put(86,19){$\cdots$}
\put(102,22){\line(1,0){5}}
\put(107,28){{\tiny $1$}}
\put(111,22){\circle*{6}}
\put(115,22){\line(1,0){5}}
\put(120,19){$\cdots$}
\put(155,28){{\tiny $1$}}
\put(135,22){\line(1,0){21}}
\put(160,22){\circle{6}}
%\put(116,10){{\tiny $-b_2$}}
%
\put(182,28){{\tiny $1$}}
\put(163,22){\line(1,0){23}}
\put(190,22){\circle{6}}
%\put(146,10){{\tiny $-b_1$}}
%
\end{picture}

\begin{picture}(400,55)(-30,0)
    \thicklines
\put(-30,20){\text{$(B_{0,n})$}}
\put(25,23){{\tiny $1$}}
\put(30,17){\circle{6}}
\put(33,17){\line(1,0){25}}
%\put(24,5){{\tiny $-b_n$}}
%
\put(55,23){{\tiny $2$}}
\put(61,17){\circle{6}}
%\put(50,5){{\tiny $-b_{n-1}$}}
 % \put(57,5){{\tiny $E_{i_0}$}}
%%
\put(65,17){\line(1,0){20}}
\put(86,14){$\cdots$}
\put(120,23){{\tiny $2$}}
\put(100,17){\line(1,0){21}}
\put(125,17){\circle{6}}
%\put(116,5){{\tiny $-b_1$}}
%
\put(147,23){{\tiny $2$}}
\put(128,17){\line(1,0){23}}
\put(155,17){\circle{6}}
%\put(148,5){{\tiny $-b$}}
%
\put(177,23){{\tiny $1$}}
\put(158,17){\line(1,0){24}}
\put(185,17){\circle*{6}}
\put(146,43){{\tiny $1$}}
\put(155,19){\line(0,1){21}}
\put(155,43){\circle{6}}
\end{picture}

\begin{picture}(400,55)(-30,0)
    \thicklines
\put(-30,20){\text{$(C_{0,n})$}}
\put(25,23){{\tiny $1$}}
\put(30,17){\circle*{6}}
\put(33,17){\line(1,0){25}}
%\put(24,5){{\tiny $-b_n$}}
%
\put(55,23){{\tiny $2$}}
\put(61,17){\circle{6}}
%\put(50,5){{\tiny $-b_{n-1}$}}
%  \put(57,5){{\tiny $E_{i_0}$}}
%%
\put(65,17){\line(1,0){20}}
\put(86,14){$\cdots$}
\put(120,23){{\tiny $2$}}
\put(100,17){\line(1,0){21}}
\put(125,17){\circle{6}}
%\put(116,5){{\tiny $-b_1$}}
%
\put(147,23){{\tiny $2$}}
\put(128,17){\line(1,0){23}}
\put(155,17){\circle{6}}
%\put(148,5){{\tiny $-b$}}
%
\put(177,23){{\tiny $1$}}
\put(158,17){\line(1,0){24}}
\put(185,17){\circle{6}}
\put(146,43){{\tiny $1$}}
\put(155,19){\line(0,1){21}}
\put(155,43){\circle{6}}
\end{picture}

\begin{picture}(400,50)(-30,0)
    \thicklines
 \put(-30,20){\text{$(D_0)$}}   
  \put(25,18){{\tiny $1$}}
\put(30,12){\circle*{6}}
\put(33,12){\line(1,0){24}}
  \put(55,18){{\tiny $2$}}
\put(60,12){\circle{6}}
\put(63,12){\line(1,0){24}}
  \put(80,18){{\tiny $3$}}
\put(90,12){\circle{6}}
\put(93,12){\line(1,0){24}}
  \put(115,18){{\tiny $2$}}
\put(120,12){\circle{6}}
\put(123,12){\line(1,0){24}}
  \put(145,18){{\tiny $1$}}
\put(150,12){\circle{6}}
\put(90,16){\line(0,1){15}}
  \put(83,38){{\tiny $2$}}
\put(90,34){\circle{6}}
 % \put(95,35){{\tiny $E_{i_0}$}}
\end{picture}

\begin{picture}(400,55)(30,0)
\thicklines
 \put(30,20){\text{$(F_0)$}}   
\put(88,23){{\tiny $1$}}
\put(97,17){\line(1,0){24}}
\put(93,17){\circle*{6}}
%\put(183,5){{\tiny $$}}
%
\put(118,23){{\tiny $2$}}
\put(127,17){\line(1,0){24}}
\put(123,17){\circle{6}}
%\put(113,5){{\tiny $$}}
%
\put(150,23){{\tiny $3$}}
\put(155,17){\circle{6}}
%\put(146,5){{\tiny $-3$}}
%
\put(177,23){{\tiny $4$}}
%\put(187,23){{$E_0$}}
\put(158,17){\line(1,0){23}}
\put(185,17){\circle{6}}
%\put(178,5){{\tiny $-3$}}
%
\put(207,23){{\tiny $3$}}
\put(188,17){\line(1,0){24}}
\put(215,17){\circle{6}}
%\put(208,5){{\tiny $-3$}}
%
\put(176,43){{\tiny $2$}}
\put(185,19){\line(0,1){21}}
\put(185,43){\circle{6}}
\put(237,23){{\tiny $2$}}
\put(218,17){\line(1,0){24}}
\put(245,17){\circle{6}}
%\put(238,5){{\tiny $-3$}}
%  \put(239,5){{\tiny $E_{i_0}$}}
%
\end{picture}
\caption{\label{fig:nG-RTP} Nearly Gorenstein RTP
}
\end{figure}
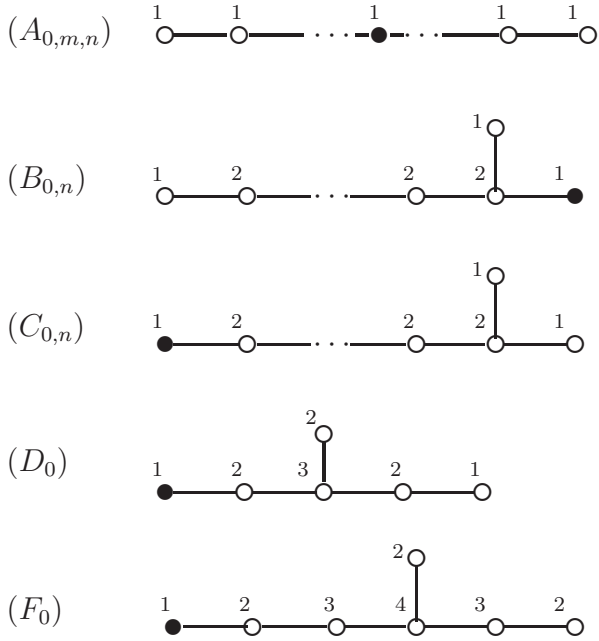
\end{ex}

\par 
Now suppose $e(A) \ge 4$ and $A$ is a nearly Gorenstein rational 
singularity. 
Then its resolution graph is not necessarily star-shaped. 
Moreover, the fundamental cycle $Z_f$ is not necessarily 
almost reduced; see the example below.

%%%%%%%%  Example 5.4
\begin{ex}\label{e:2graph}
If the weighted dual graph of $E$ is one of the graphs  in \figref{fig:2nodes},
where \begin{picture}(6,6)\thicklines\put(3,3){\circle{6}}\end{picture} (resp. \begin{picture}(6,6)\put(3,3){\circle*{6}}\end{picture}) denotes a  $(-2)$-curve (resp.  a $(-3)$-curve),  
then $\Z$ satisfies (4.b) of \thmref{t:nGrat}.   
\begin{figure}[htb]
\begin{picture}(300,55)(-20,0)
\thicklines
\put(38,23){{\tiny $1$}}
\put(47,17){\line(1,0){24}}
\put(43,17){\circle{6}}
\put(66,23){{\tiny $2$}}
\put(75,17){\circle*{6}}
\put(70,5){{\tiny $E_{i_0}$}}
\put(97,23){{\tiny $2$}}
\put(78,17){\line(1,0){23}}
\put(105,17){\circle{6}}
\put(127,23){{\tiny $1$}}
\put(108,17){\line(1,0){24}}
\put(135,17){\circle{6}}
\put(66,43){{\tiny $1$}}
\put(75,19){\line(0,1){21}}
\put(75,43){\circle{6}}
\put(96,43){{\tiny $1$}}
\put(105,19){\line(0,1){21}}
\put(105,43){\circle{6}}
\end{picture}\begin{picture}(300,55)(80,0)
\thicklines
\put(6,23){{\tiny $1$}}
\put(15,17){\line(1,0){24}}
\put(11,17){\circle{6}}
\put(38,23){{\tiny $2$}}
\put(47,17){\line(1,0){24}}
\put(43,17){\circle{6}}
\put(66,23){{\tiny $3$}}
\put(75,17){\circle{6}}
\put(97,23){{\tiny $2$}}
\put(78,17){\line(1,0){22}}
\put(105,17){\circle{6}}
\put(127,23){{\tiny $1$}}
\put(108,17){\line(1,0){24}}
\put(135,17){\circle{8}}
\put(130,5){{\tiny $E_3$}}
\put(66,43){{\tiny $2$}}
\put(75,19){\line(0,1){21}}
\put(75,43){\circle*{6}}
\put(78,35){{\tiny $E_{i_0}$}}
\put(110,43){{\tiny $1$}}
\put(78,43){\line(1,0){22}}
\put(105,43){\circle{6}}
\end{picture}
\caption{\label{fig:2nodes}  Examples of Nearly Gorenstein rings with $e(A)=4$}
\end{figure}
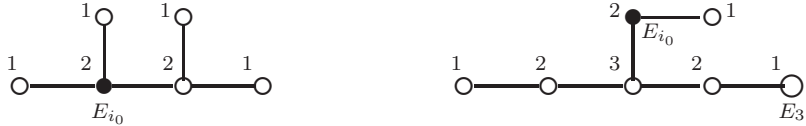
\end{ex}
%%%%%%%%%%%%%%%%%%%%%%%%%%%%%%%%%%%%%%%%%%%%%%%%%%%%%
%%%%%%%%%%%%%%%%%%%%%%%%%%%%%%%%%%%%%%%%%%%%%%%%%%%%%
%%%%%%%%%%%%%%%%%%%%%%%%%%%%%%%%%%%%%%%%%%%%%%%%%%%%%
\newpage
%%  Section 6
\section{Quotient singularities}

\par
In this section, we assume that $k = \C$ and  
we want to classify \lq\lq quotient singularities" that are nearly Gorenstein. 
Here a \lq\lq quotient singularity" means the invariant subring $\C[[x,y]]^G$, 
where $G$ is a finite subgroup of $\GL(2,\C)$.
Although the classification of quotient singularities which are nearly Gorenstein is given in \cite[Proposition 3.5]{Di},  
we give our classification in terms of the graph of $\Z$. 
\par 
In dimension $2$ and  $\chr(k)=0$, it is known that \lq\lq quotient singularities"
and \lq\lq log-terminal singularities" are equivalent and also, by  
\lq\lq mod $p$ reduction" to \lq\lq $F$-regular" rings in characteristic $p$ (cf. \cite{HW}).   
 Since such quotient singularities are certainly completion of
 graded rings, we can use the classification
 of $F$-regular  graded rings in \cite{W}. 

\par \vspace{2mm}
In what follows we assume that $A$ is \textit{not} a cyclic quotient singularity and is the localization  
of $R = R(\PP^1, D)$, with 
\[
D = E - \sum_{i=1}^r \dfrac{p_i}{q_i} P_i
\] 
with different points $P_1, \ldots, P_r$ on $\PP^1$ and we always assume that $p_i$ and $q_i$ are 
relatively prime positive integers  with $0<p_i<q_i$ for every $i$ and $E$ is an integral divisor on $\PP^1$. 
\par
If $R(\PP^1,D)$ is log-terminal (and not a cyclic quotient singularity), then $r= 3$ and 
$(q_1, q_2, q_3)$ are one of the following;  
\[
(2,2,n), \quad (2,3,3), \quad (2,3,4), \quad (2,3,5). 
\]

\par 
Now, if $\deg(E)\ge 3$, then $\Z$ is reduced and then $A$ is not nearly Gorenstein by Theorem \ref{ARnG}.  
Hence  we have $\deg E = 2$, since if $\deg E = 1$, then we will have 
$\deg D < 0$, which is absurd. 
Note that the chain \[E_1-E_2-\cdots - E_{k+1},\]
where $E_i^2 = -2$ for $1\le i \le k$ and $E_{k+1}^2 = -s$, corresponds to  
\[\dfrac{ks - (k-1)}{(k+1) s - k}. 
\] 
Since the fundamental cycle of any quotient singularity is almost reduced, 
 by Theorem \ref{ARnG} 
we conclude that nearly Gorenstein and non-Gorenstein 
quotient singularities are  classified as follows.  

%%%% Theorem 6.1
\begin{thm}[\textrm{cf. \cite{Di}}] \label{Quot-nG}
Assume $R=R(\PP^1, D)$ is a quotient singularity, not cyclic quotient singularity  
and not Gorenstein.
If $R$ is nearly Gorenstein, then $D$ is one of the following. 
\begin{enumerate}
\item $1/2 P_1 + 1/2 P_2 - \dfrac{ks - (k-1)}{(k+1) s - k} P_3$ for some $k\ge 0$ and $s\ge 3$.  
\\[1mm]
\item $1/2 P_1 + 2/3 P_2 - 1/3 P_3$
\item $1/2 P_1 + 2/3 P_2 - 2/3 P_3$
\item $1/2 P_1 + 1/3 P_2 - 1/4 P_3$
\item $1/2 P_1 + 2/3 P_2 - 1/4 P_3$
\item $1/2 P_1 + 2/3 P_2 - 3/4 P_3$
\item $1/2 P_1 + 1/3 P_2 - 1/5 P_3$
\item $1/2 P_1 + 1/3 P_2 - 3/5 P_3$
\item $1/2 P_1 + 2/3 P_2 - 1/5 P_3$
\item $1/2 P_1 + 2/3 P_2 - 3/5 P_3$
\item $1/2 P_1 + 2/3 P_2 - 4/5 P_3$
\end{enumerate}
\end{thm}

\begin{enumerate}

\item[] Type (2,2,n)  $(n \ge 1)$

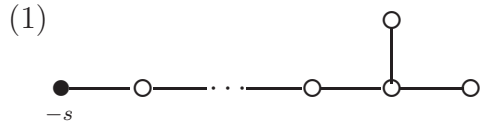
\begin{picture}(200,60)(-40,0)
    \thicklines
%\put(25,23){{\tiny $1$}}
\put(10,40){$(1)$}
\put(30,17){\circle*{6}}
\put(33,17){\line(1,0){23}}
\put(24,5){{\tiny $-s$}}
%
%\put(55,23){{\tiny $1$}}
\put(61,17){\circle{6}}
%\put(50,5){{\tiny $-b_{n-1}$}}
%
\put(65,17){\line(1,0){20}}
\put(86,14){$\cdots$}
%
%\put(120,23){{\tiny $1$}}
\put(100,17){\line(1,0){21}}
\put(125,17){\circle{6}}
%\put(116,5){{\tiny $-b_1$}}
%
%\put(147,23){{\tiny $1$}}
\put(128,17){\line(1,0){23}}
\put(155,17){\circle{6}}
%\put(148,5){{\tiny $-b$}}
%
%\put(177,23){{\tiny $1$}}
\put(158,17){\line(1,0){24}}
\put(185,17){\circle{6}}
%
%\put(146,43){{\tiny $1$}}
\put(155,19){\line(0,1){21}}
\put(155,43){\circle{6}}
\end{picture}

%%%%%%%%%%%%%%%%%%%%%%%%%%%%%%%%%%%%%%%%%%%%%%%%%%
%%%%%%%%%%%%%%%%%%%%%%%%%%%%%%%%%%%%%%%%%%%%%%%%%%
\item[] Type (2,3,3)

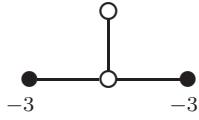
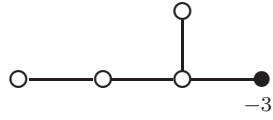
\begin{picture}(200,60)(-30,0)
    \thicklines
\put(-20,45){$(2)$}
%\put(20,23){{\tiny $1$}}
\put(25,17){\circle*{6}}
\put(16,5){{\tiny $-3$}}
%
%\put(47,23){{\tiny $1$}}
%\put(57,23){{$E_0$}}
\put(28,17){\line(1,0){23}}
\put(55,17){\circle{6}}
%\put(48,5){{\tiny $-b$}}
%
%\put(77,23){{\tiny $1$}}
\put(58,17){\line(1,0){24}}
\put(85,17){\circle*{6}}
\put(78,5){{\tiny $-3$}}
%\put(82,23){{$F$}}
%
%\put(46,43){{\tiny $1$}}
\put(55,20){\line(0,1){20}}
\put(55,43){\circle{6}}
%\end{picture}

%\begin{picture}(200,60)(-40,0)
%    \thicklines
%%%%%%%%%%%%%%%%%%%%%%%%%%%%%%%%%%%%%%%%%
\put(170,45){$(3)$}
%\put(188,23){{\tiny $1$}}
\put(197,17){\line(1,0){24}}
\put(193,17){\circle{6}}
%\put(183,5){{\tiny $$}}

%
%\put(220,23){{\tiny $1$}}
\put(225,17){\circle{6}}
%\put(216,5){{\tiny $-3$}}
%
%\put(247,23){{\tiny $1$}}
%\put(257,23){{$E_0$}}
\put(228,17){\line(1,0){23}}
\put(255,17){\circle{6}}
%\put(248,5){{\tiny $-b$}}
%
%\put(277,23){{\tiny $1$}}
\put(258,17){\line(1,0){24}}
\put(285,17){\circle*{6}}
\put(278,5){{\tiny $-3$}}
%\put(282,23){{$F$}}
%
%\put(246,43){{\tiny $1$}}
\put(255,20){\line(0,1){20}}
\put(255,43){\circle{6}}
\end{picture}

%%%%%%%%%%%%%%%%%%%%%%%%%%%%%%%%%%%%%%

%%%%%%%%%%%%%%%%%%%%%%%%%%%%%%%%%%%%%%%%
\item[] Type (2,3,4)

%%%%%%%%%%%%%%%%%%%%%%%%%%%%%%%%%%%%%%%
\begin{picture}(200,60)(60,0)
\thicklines
\put(10,45){$(4)$}
%\put(38,23){{\tiny $1$}}
\put(27,17){\line(1,0){24}}
\put(23,17){\circle{6}}
%\put(13,5){{\tiny $$}}
%
%\put(50,23){{\tiny $1$}}
\put(55,17){\circle{6}}
%\put(46,5){{\tiny $-3$}}
%
%\put(77,23){{\tiny $1$}}
%\put(87,23){{$E_0$}}
\put(58,17){\line(1,0){23}}
\put(85,17){\circle{6}}
%\put(78,5){{\tiny $-b$}}
%
%\put(77,23){{\tiny $1$}}
\put(88,17){\line(1,0){24}}
\put(115,17){\circle*{6}}
\put(108,5){{\tiny $-4$}}
%\put(113,23){{$F$}}
%
%\put(76,43){{\tiny $1$}}
\put(85,20){\line(0,1){20}}
\put(85,43){\circle{6}}
\end{picture}
\begin{picture}(200,60)(140,0)
    \thicklines
\put(10,45){$(5)$}
%\put(50,23){{\tiny $1$}}
\put(55,17){\circle*{6}}
\put(46,5){{\tiny $-3$}}
%
%\put(77,23){{\tiny $1$}}
%\put(87,23){{$E_0$}}
\put(58,17){\line(1,0){23}}
\put(85,17){\circle{6}}
%\put(78,5){{\tiny $-b$}}
%
%\put(107,23){{\tiny $1$}}
\put(88,17){\line(1,0){24}}
\put(115,17){\circle*{6}}
\put(108,5){{\tiny $-4$}}
%\put(112,23){{$F$}}
%
%\put(76,43){{\tiny $1$}}
\put(85,20){\line(0,1){20}}
\put(85,43){\circle{6}}
%\end{picture}
%\begin{picture}(200,60)(-30,0)
%    \thicklines
%%%%%%%%%%%%%%%%%%%%%%%%%%%%%%%%%%%%%%%%%%
%\put(188,23){{\tiny $1$}}
\put(150,45){$(6)$}
\put(167,17){\line(1,0){24}}
\put(163,17){\circle{6}}
%\put(153,5){{\tiny $$}}
%
%\put(188,23){{\tiny $1$}}
\put(197,17){\line(1,0){24}}
\put(193,17){\circle{6}}
%\put(183,5){{\tiny $$}}
%
%\put(220,23){{\tiny $1$}}
\put(225,17){\circle{6}}
%\put(216,5){{\tiny $-3$}}
%
%\put(247,23){{\tiny $1$}}
%\put(257,23){{$E_0$}}
\put(228,17){\line(1,0){23}}
\put(255,17){\circle{6}}
%\put(248,5){{\tiny $-b$}}
%
%\put(277,23){{\tiny $1$}}
\put(258,17){\line(1,0){24}}
\put(285,17){\circle*{6}}
\put(278,5){{\tiny $-3$}}
%\put(282,23){{$F$}}
%
%\put(246,43){{\tiny $1$}}
\put(255,20){\line(0,1){20}}
\put(255,43){\circle{6}}
\end{picture}

%%%%%%%%%%%%%%%%%%%%%%%%%%%%%%%%%%%%%%%%%%%%%
\item[] Type(2,3,5)

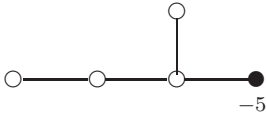
\begin{picture}(200,60)(60,0)
%    \thicklines
%%%%%%%%%%%%%%%%%%%%%%%%%%%%%%%%%%%%%%%%%%
\put(20,45){$(7)$}
%\put(38,23){{\tiny $1$}}
\put(47,17){\line(1,0){24}}
\put(43,17){\circle{6}}
%\put(33,5){{\tiny $$}}

%
%\put(70,23){{\tiny $1$}}
\put(75,17){\circle{6}}
%\put(66,5){{\tiny $-3$}}
%
%\put(97,23){{\tiny $1$}}
%\put(107,23){{$E_0$}}
\put(78,17){\line(1,0){23}}
\put(105,17){\circle{6}}
%\put(98,5){{\tiny $-b$}}
%
%\put(127,23){{\tiny $1$}}
\put(108,17){\line(1,0){24}}
\put(135,17){\circle*{6}}
\put(128,5){{\tiny $-5$}}
%\put(132,23){{$F$}}
%
%\put(116,43){{\tiny $1$}}
\put(105,19){\line(0,1){21}}
\put(105,43){\circle{6}}
\end{picture}
%%%%%%%%%%%%%%%%%%%%%%%%%%%%%%%%%%%%%%
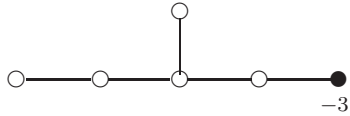
\begin{picture}(200,60)(200,0)
\put(100,45){$(8)$}
%\put(118,23){{\tiny $1$}}
\put(127,17){\line(1,0){24}}
\put(123,17){\circle{6}}
%\put(113,5){{\tiny $$}}
%
%\put(150,23){{\tiny $1$}}
\put(155,17){\circle{6}}
%\put(146,5){{\tiny $-3$}}
%
%\put(177,23){{\tiny $1$}}
%\put(187,23){{$E_0$}}
\put(158,17){\line(1,0){23}}
\put(185,17){\circle{6}}
%\put(178,5){{\tiny $-b$}}
%
%\put(207,23){{\tiny $1$}}
\put(188,17){\line(1,0){24}}
\put(215,17){\circle{6}}
%\put(208,5){{\tiny $-3$}}
%
%\put(176,43){{\tiny $1$}}
\put(185,19){\line(0,1){21}}
\put(185,43){\circle{6}}
%
%\put(237,23){{\tiny $1$}}
\put(218,17){\line(1,0){24}}
\put(245,17){\circle*{6}}
\put(238,5){{\tiny $-3$}}
%\put(242,23){{$F$}}
%\end{picture}
%\begin{picture}(200,60)(-30,0)
%    \thicklines
\put(260,45){$(9)$}
%\put(270,23){{\tiny $1$}}
\put(275,17){\circle*{6}}
\put(266,5){{\tiny $-3$}}
%
%\put(297,23){{\tiny $1$}}
%\put307,23){{$E_0$}}
\put(278,17){\line(1,0){23}}
\put(305,17){\circle{6}}
%\put(298,5){{\tiny $-b$}}
%
%\put(327,23){{\tiny $1$}}
\put(308,17){\line(1,0){24}}
\put(335,17){\circle*{6}}
\put(328,5){{\tiny $-5$}}
%\put(332,23){{$F$}}
%
%\put(296,43){{\tiny $1$}}
\put(305,19){\line(0,1){21}}
\put(305,43){\circle{6}}
\end{picture}
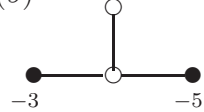
\begin{picture}(200,60)(40,0)
\thicklines
\put(10,45){$(10)$}
%\put(38,23){{\tiny $1$}}
\put(27,17){\line(1,0){24}}
\put(23,17){\circle*{6}}
\put(13,5){{\tiny $-3$}}
%
%\put(50,23){{\tiny $1$}}
\put(55,17){\circle{6}}
%\put(46,5){{\tiny $-3$}}
%
%\put(77,23){{\tiny $1$}}
%\put(87,23){{$E_0$}}
\put(58,17){\line(1,0){23}}
\put(85,17){\circle{6}}
%\put(78,5){{\tiny $-b$}}
%
%\put(27,23){{\tiny $1$}}
\put(88,17){\line(1,0){24}}
\put(115,17){\circle*{6}}
\put(108,5){{\tiny $-3$}}
%\put(112,23){{$F$}}
%
%\put(76,43){{\tiny $1$}}
\put(85,20){\line(0,1){20}}
\put(85,43){\circle{6}}
\end{picture}
%%%%%%%%%%%%%%%%%%%%%%%%%%%%%%%%%%%%%%%
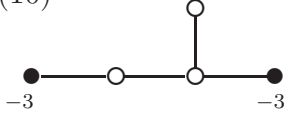
\begin{picture}(200,60)(100,0)
\thicklines
\put(40,45){$(11)$}
%\put(58,23){{\tiny $1$}}
\put(47,17){\line(1,0){24}}
\put(43,17){\circle{6}}
%\put(33,5){{\tiny $$}}
%
%\put(88,23){{\tiny $1$}}
\put(77,17){\line(1,0){24}}
\put(73,17){\circle{6}}
%\put(63,5){{\tiny $$}}
%
%\put(118,23){{\tiny $1$}}
\put(107,17){\line(1,0){24}}
\put(103,17){\circle{6}}
%\put(93,5){{\tiny $$}}
%
%\put(130,23){{\tiny $1$}}
\put(135,17){\circle{6}}
%\put(126,5){{\tiny $-3$}}
%
%\put(157,23){{\tiny $1$}}
%\put(167,23){{$E_0$}}
\put(138,17){\line(1,0){23}}
\put(165,17){\circle{6}}
%\put(158,5){{\tiny $-b$}}
%
%\put(107,23){{\tiny $1$}}
\put(168,17){\line(1,0){24}}
\put(195,17){\circle*{6}}
\put(188,5){{\tiny $-3$}}
%\put(192,23){{$F$}}
%
%\put(156,43){{\tiny $1$}}
\put(165,19){\line(0,1){21}}
\put(165,43){\circle{6}}
%
%%%%%%%%%%%%%%%%%%%%%%%%%%%%%%%%%%%%%%
\end{picture}
\end{enumerate}

\begin{proof} 
We give a sketch of the proof by using the resolution graph (8). 
Let $R=R(\PP^1, D)$ be a quotient singularity of type $(2,3,5)$
whose the resolution graph is as follows 
(see e.g., the graph $\Gamma_{13}(b)$ in 
\cite[Theorem 4.1]{MY}):  

 
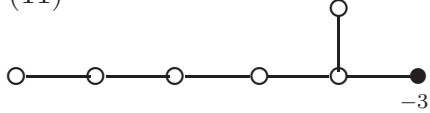
\par \vspace{2mm}
\begin{picture}(300,55)(-40,0)
\thicklines
%\put(0,40){\text{(2-A)}}
\put(8,23){{\tiny $E_{3}$}}
\put(17,17){\line(1,0){24}}
\put(13,17){\circle{6}}
%\put(12,5){{\tiny $2$}}
%
\put(38,23){{\tiny $E_{2}$}}
\put(47,17){\line(1,0){24}}
\put(43,17){\circle{6}}
%\put(40,5){{\tiny $E_{2}$}}
%
\put(66,23){{\tiny $E$}}
\put(75,17){\circle*{6}}
\put(69,5){{\tiny $-b$}}
\put(97,23){{\tiny $E_4$}}
\put(78,17){\line(1,0){23}}
\put(105,17){\circle{6}}
%\put(100,5){{\tiny $E_4$}}
%
\put(127,23){{\tiny $E_5$}}
\put(108,17){\line(1,0){24}}
\put(135,17){\circle*{6}}
\put(128,5){{\tiny $-3$}}
\put(63,47){{\tiny $E_1$}}
\put(75,19){\line(0,1){21}}
\put(75,43){\circle{6}}
%\put(80,40){{\tiny $E_1$}}  
%
%\put(157,23){{\tiny $2$}}
%\put(138,17){\line(1,0){24}}
%\put(165,17){\circle{6}}
%\put(158,5){{\tiny $-3$}}
  
%
\end{picture}
\par \vspace{2mm}
Since $\Z$ is almost reduced, we have $b=2$ by Theorem \ref{ARnG}. 
Moreover, each branch corresponds to the following Hirzebruch-Jung cotinued fraction:
\begin{eqnarray*}
E_1 & \longleftrightarrow& \frac{2}{1} \\[2mm]
E_3-E_2 & \longleftrightarrow& 2- \frac{1}{2} = \frac{3}{2} \\[2mm]
E_4-E_5 & \longleftrightarrow& 2- \frac{1}{3} = \frac{5}{3}. 
\end{eqnarray*}
Hence we obtain that 
\[
D=(P_1+P_2)-\frac{1}{2}P_1-\frac{2}{3}P_2-\frac{3}{5}P_3=\frac{1}{2}P_1+\frac{1}{3}P_2-\frac{3}{5}P_3. 
\]
The other cases can be proved similarly.   
\end{proof}

%%%  Remark 6.2
\begin{rem}
It seems (5)  in our list is missing in \cite[Proposition 3.5]{Di}. 
This example agrees with the example given in \cite[Proposition 3.1]{CS}. 
Indeed, it is $11^{th}$ Veronese subring of rational double point 
of type $(E_7)$.
\end{rem}

\par \vspace{2mm}
For quotient singularities, $\ell_A(A/\Tr_A(K_A))$ is bounded
above by $e(A)-1$.

%%%%  Proposition 6.3
\begin{prop} \label{res-quot}
Suppose that $A$ is a quotient singularity which is neither Gorenstein nor cyclic quotient singularity. 
Then we have
\[
\ell_A(A/\Tr_A(K_A))=e(A)-1-\sum_{E_i : end} (-E_i^2-2). 
\]
In particular, $\ell_A(A/\Tr_A(K_A)) \le e(A)-1$. 
\end{prop}

\begin{proof}
By assumption, the resolution graph is the following shape, where $n \ge m \ge 1$ are integers:

\par 
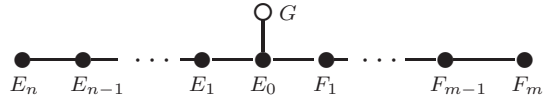
\begin{picture}(200,50)(-30,10)
    \thicklines
\put(30,22){\circle*{6}}
\put(25,10){{\tiny $E_n$}}
\put(33,22){\line(1,0){17}}
\put(53,22){\circle*{6}}
\put(48,10){{\tiny $E_{n-1}$}}
\put(56,22){\line(1,0){10}}
\put(72,19){$\cdots$}
\put(90,22){\line(1,0){5}}
\put(98,22){\circle*{6}}
\put(93,10){{\tiny $E_{1}$}}
\put(102,22){\line(1,0){15}}
\put(121,22){\circle*{6}}
\put(116,10){{\tiny $E_{0}$}}
\put(125,22){\line(1,0){15}}
\put(145,22){\circle*{6}}
\put(140,10){{\tiny $F_{1}$}}
\put(148,22){\line(1,0){5}}
\put(158,19){$\cdots$}
\put(121,25){\line(0,1){12}}
\put(121,40){\circle{6}}
\put(127,37){{\tiny $G$}}
\put(175,22){\line(1,0){11}}
\put(190,22){\circle*{6}}
\put(184,10){{\tiny $F_{m-1}$}}
\put(193,22){\line(1,0){23}}
\put(220,22){\circle*{6}}
\put(215,10){{\tiny $F_{m}$}}
\end{picture}
\par \vspace{2mm}
We put $b_i=-E_i^2$ and $c_j=-F_j^2$ for each $i$ and $j$. 
In order to prove the formula, we may assume that $A$ is \textit{not}
nearly Gorenstein. 
Moreover, when $e(A) \le 2$, the assertion is true (cf. \cite{MY}). 
So, it is enough to consider the following $3$ cases. 
\par 
Since $p_a(\Z)=0$, we have 
\begin{equation}
e(A)=-\Z^2=2+K_X\Z=2+\sum_{i=0}^n (b_i-2)+\sum_{j=1}^m (c_j-2). 
\end{equation}
\begin{flushleft}
{\bf Case 1: The case where $b_0 \ge 3$. }
\end{flushleft} 
\par 
Note that $\Z$ is reduced. 
If $Y=E_{n-1}+\cdots+E_1+E_0+F_1+\cdots+F_{m-1}$, then 
one can show that $K_X+\Z+Y$ is anti-nef and thus $F=\Z+Y$. 
Since $Y$ can be regarded as the fundamental cycle on $\Supp(Y)$, 
$p_a(Y)=0$.    
Then by Riemann-Roch formula, we get 
\[
\ell_A(A/\Tr_A(K_A))
= 1-p_a(F) 
= 1-p_a(\Z)-p_a(Y)-Y\Z+1 
= 2-Y\Z. 
\]
On the other hand, direct computation yields 
\[
2-Y\Z=2-\sum_{i=1}^{n-1}(-b_i+2)-(-b_0+3)-\sum_{j=1}^{m-1}(-c_j+2)
=e(A)-1-(b_n-2)-(c_m-2)
\]
Thus we obtain the required formula in this case. 
\begin{flushleft}
{\bf Case 2: The case where $b_0 =2$, $m=1$ and $c_1=2$.}
\end{flushleft} 

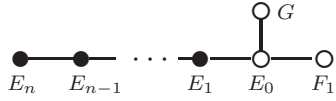
\begin{picture}(200,50)(-30,10)
    \thicklines
\put(30,22){\circle*{6}}
\put(25,10){{\tiny $E_n$}}
\put(33,22){\line(1,0){17}}
\put(53,22){\circle*{6}}
\put(48,10){{\tiny $E_{n-1}$}}
\put(56,22){\line(1,0){10}}
\put(72,19){$\cdots$}
\put(90,22){\line(1,0){5}}
\put(98,22){\circle*{6}}
\put(93,10){{\tiny $E_{1}$}}
\put(102,22){\line(1,0){15}}
\put(121,22){\circle{6}}
\put(116,10){{\tiny $E_{0}$}}
\put(125,22){\line(1,0){15}}
\put(145,22){\circle{6}}
\put(140,10){{\tiny $F_{1}$}}
%\put(148,22){\line(1,0){5}}
%\put(158,19){$\cdots$}
%
\put(121,25){\line(0,1){12}}
\put(121,40){\circle{6}}
\put(127,37){{\tiny $G$}}
\end{picture}

\par \vspace{3mm}
As $A$ is not a nearly Gorenstein, there exists an integer $k$ with $1 \le k < n$ such that $b_k \ge 3$. 
We may assume $k$ is the least integer among those integers $k$. 
Then 
\[
\Z=2(E_0+E_1+\cdots+E_{k-1})+E_k+\cdots+E_n+F_1+G. 
\]
If we put $Y=E_k+E_{k+1}+\cdots + E_{n-1}$, one can show that 
$F=\Z+Y$. 
Then by Riemann-Roch formula, we get 
\[
\ell_A(A/\Tr_A(K_A))
= 1-p_a(F) 
= 1-p_a(\Z)-p_a(Y)-Y\Z+1 
= 2-Y\Z. 
\]
On the other hand, direct computation yields 
\[
2-Y\Z=2-\sum_{i=k+1}^{n-1}(-b_i+2)-(-b_k+3)
=e(A)-1-(b_n-2)-(c_1-2). 
\]
Thus we obtain the required formula in this case.
\begin{flushleft}
{\bf Case 3: The case where $b_0 =2$ and \lq\lq $m \ge 2$ or $c_1 \ge 3$''.} 
\end{flushleft} 
It is enough to consider the following two graphs because 
$A$ is not nearly Gorenstein. 

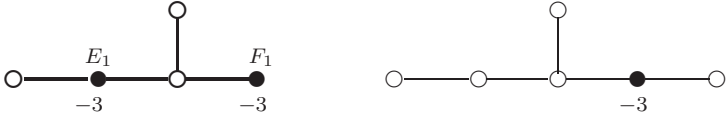
\begin{picture}(200,60)(0,0)
\thicklines
%\put(10,45){$(10)$}
%\put(38,23){{\tiny $1$}}
\put(27,17){\line(1,0){24}}
\put(23,17){\circle{6}}
%\put(13,5){{\tiny $-3$}}
%
%\put(50,23){{\tiny $1$}}
\put(50,23){{\tiny $E_1$}}
\put(55,17){\circle*{6}}
\put(46,5){{\tiny $-3$}}
%
%\put(77,23){{\tiny $1$}}
%\put(87,23){{$E_0$}}
\put(58,17){\line(1,0){23}}
\put(85,17){\circle{6}}
%\put(78,5){{\tiny $-b$}}
%
%\put(27,23){{\tiny $1$}}
\put(88,17){\line(1,0){24}}
\put(115,17){\circle*{6}}
\put(108,5){{\tiny $-3$}}
\put(112,23){{\tiny $F_1$}}
%
%\put(76,43){{\tiny $1$}}
\put(85,20){\line(0,1){20}}
\put(85,43){\circle{6}}
\end{picture}
%%%%%%%%%%%%%%%%%%%%%%%%%%%%%
\begin{picture}(200,60)(160,0)
%\put(100,45){$(8)$}
%\put(118,23){{\tiny $1$}}
\put(127,17){\line(1,0){24}}
\put(123,17){\circle{6}}
%\put(113,5){{\tiny $$}}
%
%\put(150,23){{\tiny $1$}}
\put(155,17){\circle{6}}
%\put(146,5){{\tiny $-3$}}
%
%\put(177,23){{\tiny $1$}}
%\put(187,23){{$E_0$}}
\put(158,17){\line(1,0){23}}
\put(185,17){\circle{6}}
%\put(178,5){{\tiny $-b$}}
%
%\put(207,23){{\tiny $1$}}
\put(188,17){\line(1,0){24}}
\put(215,17){\circle*{6}}
\put(208,5){{\tiny $-3$}}
%
%\put(176,43){{\tiny $1$}}
\put(185,19){\line(0,1){21}}
\put(185,43){\circle{6}}
%
%\put(237,23){{\tiny $1$}}
\put(218,17){\line(1,0){24}}
\put(245,17){\circle{6}}
%\put(238,5){{\tiny $-3$}}
%\put(242,23){{\tiny $F_1$}}
\end{picture}
\par \vspace{1mm}
In the first graph, we have $e(A)=4$. 
One can easily see that $F=Z+E_1$. 
Then $\ell_A(A/\Tr_A(K_A))=2=e(A)-1-(c_1-2)$. 
\par \vspace{1mm}
In the second graph, $A$ is a rational triple point. 
Then $\ell_A(A/\Tr_A(K_A))=2=e(A)-1$. 
\par \vspace{2mm}
In all cases, we proved the required formula. 
\end{proof}

\par \vspace{2mm}
If $A$ is a rational singularity, 
$\ell_A(A/\Tr_A(K_A))$ has no upper bound even if we fix $e(A)$.  

%%%%  Example 6.4
\begin{ex}[see \cite{MY}] \label{MYTypeA}
Suppose that $\ell,m,n$ are integers with $\ell,m,n \ge 2$. 
Put \[
A=\mathbb{C}[x,y,z,t]_{(x,y,z,t)}/(xy-t^{\ell+m+2},\, xz-t^{n+2}-zt^{\ell+1},\,yz+yt^{n+1}-zt^{m+1}),  
\]
a rational triple point of type $(A_{\ell,m,n})$. 
Then we have 
\[
\ell_A(A/\Tr_A(K_A))=\min\{\ell,m,n\}+1. 
\]

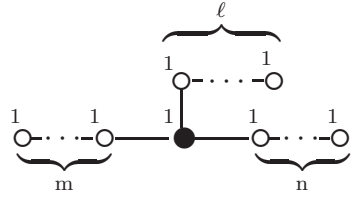
\begin{figure}[htb]
\begin{picture}(200,60)(-40,0)
    \thicklines
\put(25,18){{\tiny $1$}}
\put(30,12){\circle{6}}
\put(33,12){\line(1,0){4}}
\put(28,7){$\underbrace{\phantom{AAAA}}_{\text {m}}$}
\put(38,9){$\cdots$}
\put(52,12){\line(1,0){4}}
  \put(55,18){{\tiny $1$}}
\put(61,12){\circle{6}}
\put(65,12){\line(1,0){20}}
  \put(83,18){{\tiny $1$}}
\put(91,12){\circle*{8}}
\put(90,16){\line(0,1){15}}
\put(95,12){\line(1,0){20}}
  \put(115,18){{\tiny $1$}}
\put(120,12){\circle{6}}
  \put(145,18){{\tiny $1$}}
\put(150,12){\circle{6}}
\put(128,9){$\cdots$}
\put(123,12){\line(1,0){4}}
\put(142,12){\line(1,0){4}}
\put(118,7){$\underbrace{\phantom{AAAA}}_{\text {n}}$}
\put(83,38){{\tiny $1$}}
\put(90,34){\circle{6}}
  \put(120,40){{\tiny $1$}}
\put(125,34){\circle{6}}
\put(100,31){$\cdots$}
\put(94,34){\line(1,0){4}}
\put(117,34){\line(1,0){4}}
\put(83,40){$\overbrace{\phantom{AAAAA}}^{\text {$\ell$}}$}
\end{picture}
\caption{\label{fig:AtypeRTP}  Rational Triple Point of type $A_{\ell,m,n}$
}
\end{figure}
\end{ex}

%%%%%%  Example 6.5
\begin{ex}\label{e=4}  
Let $R=R( \mathbb{P}^1, D)$ with $D = (1/n) (P_1+P_2) + 2/(2n+1) P_3$ for 
$n\ge 2$. 
If we put $P_1 = \infty, P_2 = 0, P_3 = 1$, then we have; 
\[
R = k [T, x^{-1} T^{n+1},\, (x-1)^{-1} T^{n+1},\, x T^{n+1}, \,x^2 T^{2n+1} ].
\] 
The minimal resolution of $R$ is a star-shaped graph with central curve $E_0$ with 
$E_0^2= -3$ and $3$ branches $E_{i,1}-E_{i,2}-\cdots - E_{i, n-1}$ for $i=1,2,3$ 
with $E_{i,j}^2 = -2$ for all $i,j$ except that $E_{3,n-1}^2 = -3$. 
Then we see that $\Z$ is reduced and  
$F = n E_0 + (n-1)\sum_{i=1}^3 E_{i,1} + (n-2)\sum_{i=1}^3 E_{i,2}+ 
\cdots + \sum_{i=1}^3 E_{i,n-1}$. \par
Since $F^2 = - 3n -1$ and $K_X F = n+1$, we see that 
\[
\ell_A( A/\Tr K_A) = \ell_A (A/ H^0(X, \cO_X(-F)) = n.
\]    
\end{ex}

\begin{figure}[htb]
\begin{picture}(300,55)(-20,0)

\thicklines
%
%\put(6,23){{\tiny $1$}}
\put(4,17){\line(1,0){10}}
\put(1,17){\circle{6}}
\put(-5,5){{\tiny $E_{2,n-1}$}}
\put(15,14){$\cdots$}
\put(30,17){\line(1,0){10}}
%
%\put(38,23){{\tiny $2$}}
\put(47,17){\line(1,0){24}}
\put(43,17){\circle{6}}
\put(38,5){{\tiny $E_{2,1}$}}
%
%\put(66,23){{\tiny $3$}}
\put(75,17){\circle*{6}}
\put(70,5){{\tiny $E_{0}$}}
%
%\put(97,23){{\tiny $2$}}
\put(78,17){\line(1,0){23}}
\put(105,17){\circle{6}}
\put(100,5){{\tiny $E_{1,1}$}}
\put(108,17){\line(1,0){15}}
\put(128,14){$\cdots$}
%
%\put(127,23){{\tiny $1$}}
\put(151,17){\line(1,0){10}}
\put(165,17){\circle{6}}
\put(160,5){{\tiny $E_{1,n-1}$}}
%
%\put(66,43){{\tiny $2$}}
\put(75,20){\line(0,1){20}}
\put(75,43){\circle{6}}
\put(70,49){{\tiny $E_{3,1}$}}
%\put(78,35){{\tiny $E_{i_0}$}}
%
%\put(110,43){{\tiny $1$}}
\put(78,43){\line(1,0){22}}
\put(105,43){\circle{6}}
\put(100,49){{\tiny $E_{3,2}$}}
\put(108,43){\line(1,0){15}}
\put(128,40){$\cdots$}
%
%\put(110,43){{\tiny $1$}}
\put(151,43){\line(1,0){10}}
\put(165,43){\circle*{6}}
\put(160,49){{\tiny $E_{3,n-1}$}}
\end{picture}
\caption{\label{fig:gradedRat}  Graded rational singularity of 
$\ell_A(A/\Tr K_A)=n$}
\end{figure}
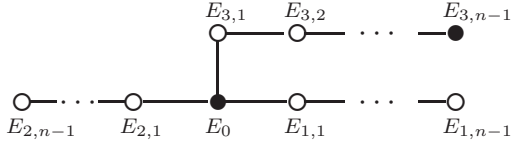

\par \vspace{2mm}
We recall the definition of Ulrich ideals. 

%%%%%  Definition 6.6
\begin{defn}[\textrm{cf. \cite{GOTWY1, GOTWY2}}]
Let $I$ be an $\m$-primary ideal of $A$ which is not a parameter ideal. 
The ideal $I$ is called an \textit{Ulrich ideal} if there exists a minimal 
reduction $\frq$ of $I$ such that $I^2=\frq I$ and $I/I^2$ 
is a free $A/I$-module. 
\end{defn}

\par \vspace{2mm}
Assume that $A$ is non-Gorenstein rational singularity and 
let $I \subset A$ 
be an $\m$-primary ideal of $A$. 
Let $\mathcal{X}_A$ denote the set of Ulrich ideals of $A$. 
Then 
\begin{enumerate}
\item
Any Ulrich ideal $I$ is an integrally closed ideal which 
is represented on the minimal resolution of singularities; see \cite[Section 6]{GOTWY2}.
\item $I \in \mathcal{X}_A$ if and only if 
$e(A)=(\mu_A(I)-1)\ell_A(A/I)$; see \cite[Lemma 2.3]{GOTWY1}.
\item If $I \subset A$ is an Ulrich ideal, then $\Tr_A(K_A) \subset I$; 
see \cite[Proposition 2.12, Corollary 2.14]{GK}. 
In particular, if $A$ is nearly Gorenstein, then $\mathcal{X}_A =\{\m\}$. 
\end{enumerate}

\par 
If $A$ is a rational triple point, then 
$\Tr_A(K_A)$ is an Ulrich ideal and 
$\mathcal{X}_A=\{I \subsetneq A \,|\, I \supset \Tr_A(K_A)\}$;
see \cite{MY} for details. 
So it is natural to ask the following question. 

%%%% Question 6.7
\begin{quest}
Assume that $A$ is a rational singularity. 
\begin{enumerate}
\item Is $\Tr_A(K_A)$ an Ulrich ideal? 
\item If $\mathcal{X}_A=\{\m\}$ then is $A$ nearly Gorenstein?
\end{enumerate}
\end{quest}

\par
However, if $e(A) \ge 4$, then this is \textit{not} true in general. 

%%%% Example 6.8
\begin{ex} \label{ex: nGneUG}
If the resolution graph is the following, 
then $A$ is a quotient singularity with $e(A)=4$. 
Moreover, it is not a nearly Gorenstein ring, 
but $\mathcal{X}_A = \{\m\}$.  
Indeed, $\ell_A(A/\Tr_A(K_A))=2$. 
In particular, $\Tr_A(K_A)$ is \textit{not} an Ulrich ideal.  

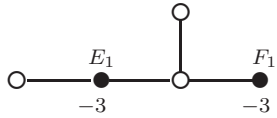
\begin{figure}[htb]
\begin{picture}(200,60)(0,0)
\thicklines
%\put(10,45){$(10)$}
%\put(38,23){{\tiny $1$}}
\put(27,17){\line(1,0){24}}
\put(23,17){\circle{6}}
%\put(13,5){{\tiny $-3$}}
%
%\put(50,23){{\tiny $1$}}
\put(50,23){{\tiny $E_1$}}
\put(55,17){\circle*{6}}
\put(46,5){{\tiny $-3$}}
%
%\put(77,23){{\tiny $1$}}
%\put(87,23){{$E_0$}}
\put(58,17){\line(1,0){23}}
\put(85,17){\circle{6}}
%\put(78,5){{\tiny $-b$}}
%
%\put(27,23){{\tiny $1$}}
\put(88,17){\line(1,0){24}}
\put(115,17){\circle*{6}}
\put(108,5){{\tiny $-3$}}
\put(112,23){{\tiny $F_1$}}
%
%\put(76,43){{\tiny $1$}}
\put(85,20){\line(0,1){20}}
\put(85,43){\circle{6}}
\end{picture}
\caption{\label{fig:notUlrich}  $\tr_A(K_A)$ is not Ulrich} 
\end{figure}
\end{ex}

%%%%%%%%%%%%%%%%%%%%%%%
%%%%%%%%%%%%%%%%%%%%%%%%%%%%%%%%%%%%%%%%%%%%%%%%%%%%%%%%%%%%

\end{document}